\documentclass[a4paper, twoside]{scrartcl}

\usepackage{amsfonts,amssymb,amsmath,amsthm,amstext,amssymb,mathtools}
\usepackage{subcaption,float,color,graphicx,enumitem}
\usepackage{dsfont}  
\usepackage[normalem]{ulem} 
\usepackage[authoryear,sort,round]{natbib}  
\usepackage{fullpage} 
\usepackage{hyperref,nicefrac}
\hypersetup{
	colorlinks,
	linkcolor=blue,
	citecolor=blue,
	urlcolor=blue,
}
\newcommand*{\arXiv}[1]{\bgroup\color{blue}\href{https://arxiv.org/abs/#1}{arXiv:#1}\egroup}
\newcommand*{\doi}[1]{\bgroup\color{blue}\href{https://doi.org/#1}{doi:#1}\egroup}
\newcommand*{\email}[1]{\bgroup\color{blue}\href{mailto:#1}{#1}\egroup}
\renewcommand*{\url}[1]{\bgroup\color{blue}\href{#1}{#1}\egroup}
\usepackage{enumitem, moreenum}
\setlist[enumerate]{nosep}
\setlist[itemize]{nosep}
\usepackage{mleftright} \mleftright
\renewcommand{\qedsymbol}{$\blacksquare$}
\renewenvironment{proof}[1][\proofname]{\noindent{\bfseries\sffamily #1.} }{\hfill\qedsymbol\medskip}

\usepackage[table]{xcolor}
\usepackage{dsfont}

\usepackage{adjustbox}
\usepackage{tikz}
\usetikzlibrary{bayesnet,er,trees,shapes.symbols,mindmap,arrows,decorations,shapes.misc,shapes.arrows,chains,matrix,positioning,scopes,decorations.pathmorphing,patterns}
\usepackage{tikz-cd} 

\usepackage{etex}
\usepackage[a4paper, top=88pt, bottom=88pt, left=72pt, right=72pt, headsep=16pt, footskip=28pt]{geometry}
\usepackage{bbm}

\usepackage{todonotes}
\usepackage[table]{xcolor}
\usepackage{booktabs}
\usepackage{tabularx}
\usepackage{makecell}
\usepackage[labelfont={sf,bf}]{caption}
\usepackage{scrlayer-scrpage, xhfill}

\usepackage{aliascnt}

\usepackage[noabbrev,capitalise,nosort,nameinlink]{cleveref}

\newcommand{\E}{\mathbb{E}}
\newcommand{\N}{\mathbb{N}}
\newcommand{\R}{\mathbb{R}}
\newcommand{\I}{\mathcal{I}}

\def\P{\mathbb{P}}
\renewcommand{\d}{\mathrm{d}}

\def\KL{D_{\operatorname{KL}}}

\DeclareMathOperator{\Var}{Var}

\definecolor{sw}{rgb}{0.8,0.1,0.1}

\definecolor{ca}{rgb}{0,0,1}

\definecolor{lt}{rgb}{1,0.53,0.0}

\newcommand*{\bR}{\mathbb{R}}

\newcommand*{\bN}{\mathbb{N}}

\newcommand*{\bE}{\mathbb{E}}
\newcommand*{\bV}{\mathbb{V}}
\newcommand*{\bP}{\mathbb{P}}

\newcommand*{\cB}{\mathcal{B}}
\newcommand*{\cN}{\mathcal{N}}

\newcommand*{\cF}{\mathcal{F}}

\newcommand*{\cP}{\mathcal{P}}

\makeatletter
\newcommand*\quark{\mathpalette\quark@{.5}}
\newcommand*\quark@[2]{\mathbin{\vcenter{\hbox{\scalebox{#2}{$\; \m@th#1\bullet \;$}}}}}
\makeatother

\makeatletter
\newcommand{\mylabel}[2]{#2\def\@currentlabel{#2}\label{#1}}
\makeatother

\newcommand*{\rd}{\mathrm{d}}

\DeclarePairedDelimiterX{\infdivx}[2]{(}{)}{#1\;\delimsize\|\;#2}

\DeclareMathOperator*{\MISE}{\mathsf{MISE}}
\DeclareMathOperator*{\MIAE}{\mathsf{MIAE}}
\DeclareMathOperator*{\IAE}{\mathsf{IAE}}

\definecolor{darkgreen}{rgb}{0,0.4,0}

\DeclarePairedDelimiter\abs{\lvert}{\rvert}%
\DeclarePairedDelimiter\mynorm{\lVert}{\rVert}%
\makeatletter
\let\oldabs\abs
\def\abs{\@ifstar{\oldabs}{\oldabs*}}
\let\oldnorm\mynorm
\def\mynorm{\@ifstar{\oldnorm}{\oldnorm*}}
\makeatother

\theoremstyle{plain}

\newtheorem{theorem}{Theorem}[section]

\newaliascnt{proposition}{theorem}
\newtheorem{proposition}[proposition]{Proposition}
\aliascntresetthe{proposition}

\newaliascnt{lemma}{theorem}
\newtheorem{lemma}[lemma]{Lemma}
\aliascntresetthe{lemma}

\newaliascnt{corollary}{theorem}
\newtheorem{corollary}[corollary]{Corollary}
\aliascntresetthe{corollary}

\theoremstyle{definition}

\newaliascnt{definition}{theorem}

\aliascntresetthe{definition}

\newaliascnt{assumption}{theorem}
\newtheorem{assumption}[assumption]{Assumption}
\aliascntresetthe{assumption}

\newaliascnt{remark}{theorem}
\newtheorem{remark}[remark]{Remark}
\aliascntresetthe{remark}

\newaliascnt{example}{theorem}

\aliascntresetthe{example}

\newaliascnt{condition}{theorem}

\aliascntresetthe{condition}

\newaliascnt{algorithm}{theorem}

\aliascntresetthe{algorithm}

\newaliascnt{notation}{theorem}

\aliascntresetthe{notation}

\newaliascnt{conjecture}{theorem}

\aliascntresetthe{conjecture}

\newaliascnt{approach}{theorem}

\aliascntresetthe{approach}

\newaliascnt{counterexample}{theorem}

\aliascntresetthe{counterexample}

\newaliascnt{problem}{theorem}

\aliascntresetthe{problem}

\newaliascnt{openproblem}{theorem}

\aliascntresetthe{openproblem}

\newaliascnt{axiom}{theorem}

\aliascntresetthe{axiom}

\newaliascnt{question}{theorem}

\aliascntresetthe{question}

\crefname{theorem}{Theorem}{Theorems}
\Crefname{theorem}{Theorem}{Theorems}

\crefname{proposition}{Proposition}{Propositions}
\Crefname{proposition}{Proposition}{Propositions}

\crefname{lemma}{Lemma}{Lemmas}
\Crefname{lemma}{Lemma}{Lemmas}

\crefname{corollary}{Corollary}{Corollaries}
\Crefname{corollary}{Corollary}{Corollaries}

\crefname{definition}{Definition}{Definitions}
\Crefname{definition}{Definition}{Definitions}

\crefname{assumption}{Assumption}{Assumptions}
\Crefname{assumption}{Assumption}{Assumptions}

\crefname{remark}{Remark}{Remarks}
\Crefname{remark}{Remark}{Remarks}

\crefname{condition}{Condition}{Conditions}
\Crefname{condition}{Condition}{Conditions}

\crefname{algorithm}{Algorithm}{Algorithms}
\Crefname{algorithm}{Algorithm}{Algorithms}

\crefname{example}{Example}{Examples}
\Crefname{example}{Example}{Examples}

\crefname{notation}{Notation}{Notations}
\Crefname{notation}{Notation}{Notations}

\crefname{openproblem}{Open Problem}{Open Problems}
\Crefname{openproblem}{Open Problem}{Open Problems}

\crefname{problem}{Problem}{Problems}
\Crefname{problem}{Problem}{Problems}

\crefname{conjecture}{Conjecture}{Conjectures}
\Crefname{conjecture}{Conjecture}{Conjectures}

\crefname{counterexample}{Counterexample}{Counterexamples}
\Crefname{counterexample}{Counterexample}{Counterexamples}

\crefname{axiom}{Axiom}{Axioms}
\Crefname{axiom}{Axiom}{Axioms}

\crefname{question}{Question}{Questions}
\Crefname{question}{Question}{Questions}

\newcommand{\absval}[1]{\lvert #1 \rvert}

\newcommand{\norm}[1]{\lVert #1 \rVert}

\newcommand{\Norm}[1]{\left\Vert #1 \right\Vert}

\numberwithin{equation}{section}
\numberwithin{figure}{section}
\numberwithin{table}{section}

\newcolumntype{C}[1]{>{\centering\arraybackslash}m{#1}}

\automark[section]{section}
\setkomafont{pageheadfoot}{\normalcolor\sffamily}
\setkomafont{pagenumber}{\normalfont\normalsize\sffamily}
\clearpairofpagestyles
\let\oldtitle\title
\renewcommand{\title}[1]{\oldtitle{#1}\newcommand{\theshorttitle}{#1}}
\newcommand{\shorttitle}[1]{\renewcommand{\theshorttitle}{#1}}
\let\oldauthor\author
\renewcommand{\author}[1]{\oldauthor{#1}\newcommand{\theshortauthor}{#1}}
\newcommand{\shortauthor}[1]{\renewcommand{\theshortauthor}{#1}}
\cohead{\xrfill[0.525ex]{0.6pt}~\theshorttitle~\xrfill[0.525ex]{0.6pt}}
\cehead{\xrfill[0.525ex]{0.6pt}~\theshortauthor~\xrfill[0.525ex]{0.6pt}}
\newcommand{\theabstract}[1]{\par\bgroup\noindent\textbf{\textsf{Abstract.}} #1\egroup}
\newcommand{\thekeywords}[1]{\par\smallskip\bgroup\noindent\textbf{\textsf{Keywords.}}\newcommand{\and}{ $\bullet$ } #1\egroup}
\newcommand{\themsc}[1]{\par\smallskip\bgroup\noindent\textbf{\textsf{2020 Mathematics Subject Classification.}}\newcommand{\and}{ $\bullet$ } #1\egroup}

\newcommand*{\affilref}[1]{\ref{affiliation#1}}
\newcommand*{\affiliation}[3]{
	\footnotetext[#1]{\label{affiliation#2}#3}
}

\usepackage{adjustbox}
\usepackage{tikz}
\usetikzlibrary{bayesnet,er,trees,shapes.symbols,mindmap,arrows,decorations,shapes.misc,shapes.arrows,chains,matrix,positioning,scopes,decorations.pathmorphing,patterns}
\usepackage{tikz-cd} 

\definecolor{nisblue}{RGB}{220,235,250}
\definecolor{nisgreen}{RGB}{220,245,225}
\definecolor{nisorange}{RGB}{255,235,210}

\setlist{topsep=0.3ex, itemsep=0.3ex}

\title{Error Bounds for Importance Sampling with Estimated Proposal Distributions}
\shorttitle{Error Bounds for Importance Sampling with Estimated Proposal Distributions}
\author{%
	Cathrine~Aeckerle-Willems\textsuperscript{\affilref{UM_old}}%
	\and
	Ilja~Klebanov\textsuperscript{\affilref{FUB}}%
	\and
	Simon~Weissmann\textsuperscript{\affilref{UM}}%
}

\shortauthor{C.~Aeckerle-Willems, I.~Klebanov, S.~Weissmann}
\date{\today}

\begin{document}
\maketitle
\affiliation{1}{UM_old}{Independent Researcher, formerly at Universit\"at Mannheim, D-68131 Mannheim, Germany (\email{cathrine.aeckerle@gmail.com})}
\affiliation{2}{FUB}{Freie Universit{\"a}t Berlin, Arnimallee 6, 14195 Berlin, Germany (\email{klebanov@zedat.fu-berlin.de})}
\affiliation{3}{UM}{Universit\"at Mannheim, D-68131 Mannheim, Germany (\email{simon.weissmann@uni-mannheim.de})}

\begin{abstract}\small
	\theabstract{
    Importance sampling with data-driven proposal distributions is widely used in practice. A common workflow first generates an auxiliary sample of size $N$ from an approximation of the target distribution, constructs a density estimate $\hat q$ such as a kernel density estimator (KDE), and then draws $n$ importance samples from this learned proposal. Despite its practical relevance, the theoretical properties of this hierarchical procedure remain poorly understood, since classical importance sampling theory assumes a fixed proposal.

    We address this gap by deriving non-asymptotic error bounds for standard, defensive, and self-normalized importance sampling estimators with random proposals. Our results separate the Monte Carlo error, scaling as $n^{-1/2}$, from the proposal approximation error measured through the mean integrated absolute and squared errors ($\MIAE$ and $\MISE$) of $\hat q$. To obtain explicit convergence rates in $(N,n)$, we establish $\MIAE$ and $\MISE$ bounds for KDEs constructed from geometrically ergodic Markov chains in stationary and non-stationary regimes. Combining these results yields quantitative guarantees for importance sampling with KDE-based proposals. Our theory provides practical guidance for selecting defensive mixture weights in a nonparametric importance sampling framework.
    }  
	
	\thekeywords{{Importance sampling}%
        \and
        {data-driven proposal distributions}
        \and
        {non-asymptotic error bounds}
        \and
        {defensive mixtures}
        \and%
		{kernel density estimation}%
		\and%
		{Markov chain Monte Carlo}%
	}

	\themsc{{62G07}%
		\and%
		{65C05}%
		\and%
		{65C40}%
        \and%
        {62G05}%
	}
\end{abstract}

\section{Introduction}
\label{sec:Introduction}

A central task in Bayesian inference is the computation of expectations under a posterior distribution.
Let $P\in\cP(\R^d)$ be a target probability measure with strictly positive Lebesgue density $p$, and let
\[
\I(f)\coloneqq \E_{P}[f] = \int_{\R^d} f(x)\, p(x) \, \d x
\]
for a measurable function $f$.
In many realistic models, direct sampling from $P$ is infeasible: the density is typically known only up to a normalizing constant, may concentrate on low-dimensional regions, and can be multi-modal.
Markov chain Monte Carlo (MCMC) provides a general-purpose solution \citep{Robert2004MC,Brooks2011MCMC}, but the resulting dependence and potentially slow mixing can make it expensive to estimate expectations accurately and reliably quantify their uncertainty.

Importance sampling (IS; \citep{owen_monte_2013,Robert2004MC}) provides an alternative approach in which expectations under $P$ are estimated using samples drawn from a proposal distribution $Q$ and reweighted by the Radon--Nikodym derivative $\rho = \frac{\rd P}{\rd Q}$.
Given independent and identically distributed (i.i.d.) samples $X_1,\dots,X_n\sim Q$, the estimator
\[
\I_n(f)\coloneqq \frac1n\sum_{i=1}^n \rho(X_i)\,f(X_i)
\]
is unbiased for $\I(f)$ under $P\ll Q$ and appropriate integrability.
In many applications, the normalizing constant of $p$ is unknown, and one instead uses the self-normalized importance sampling (SNIS) estimator based on normalized importance weights \citep{owen_monte_2013}.
The practical success of IS (and SNIS) hinges on the choice of $Q$: if $Q$ assigns too little mass to regions where $p$ is large, importance weights can become extremely variable and the estimator can be unstable or even have infinite variance \citep{owen_monte_2013,veach1995optimally}.
This difficulty is especially pronounced in Bayesian problems, where the posterior mass may be highly concentrated.

\begin{figure}[t]
\centering
\begin{tikzpicture}[
    scale=0.96,
    transform shape,
    node distance=2cm,
    box/.style={
        draw,
        rounded corners,
        align=center,
        minimum height=1.1cm,
        minimum width=2cm,
        font=\small
    },
    mainbox/.style={box, fill=nisblue},
    auxbox/.style={box, fill=nisgreen},
    defbox/.style={box, fill=nisorange},
    arrow/.style={->, thick},
    optarrow/.style={->, thick, dashed}
]

\node[mainbox] (target) {\shortstack{Target $P$,\\ density $p$}};

\node[auxbox, below = 1.3cm of target] (aux)
{\shortstack{Auxiliary\\sample $Z_{1:N}$\\(iid/MCMC)}};

\node[defbox, right=of aux] (kde)
{\shortstack{KDE\\proposal\\$\hat q = \hat q_N$}};

\node[auxbox, right=of kde] (is)
{\shortstack{IS sample\\$X_{1:n}\sim\hat Q$}};

\node[mainbox, right=of is] (est)
{\shortstack{IS estimator\\$\mathcal I_n(f)$}};

\node[defbox, above=1.3cm of kde] (def)
{\shortstack{Defensive\\mixture\\$\hat q = \hat q_{N,\delta_N}$}};

\draw[arrow] (target) -- node[left, align=center, font=\small]{approx.\\draws} (aux);
\draw[arrow] (aux) -- node[below, align=center, font=\small]{density\\estimation} (kde);
\draw[arrow] (kde) -- node[below, font=\small]{sampling} (is);
\draw[optarrow] (def) -- node[above right, font=\small]{sampling} (is);
\draw[arrow] (is) -- node[below, font=\small]{estimation} (est);

\draw[optarrow] (kde) -- node[left, align=center, font=\small]{optional\\stabilization} (def);

\end{tikzpicture}

\caption{Two-stage importance sampling pipeline with estimated proposal distribution. Auxiliary samples are used to learn a proposal via kernel density estimation. Optionally, a defensive mixture stabilizes the proposal before importance sampling and estimation of $I(f)$.}
\label{fig:nis_pipeline}
\end{figure}
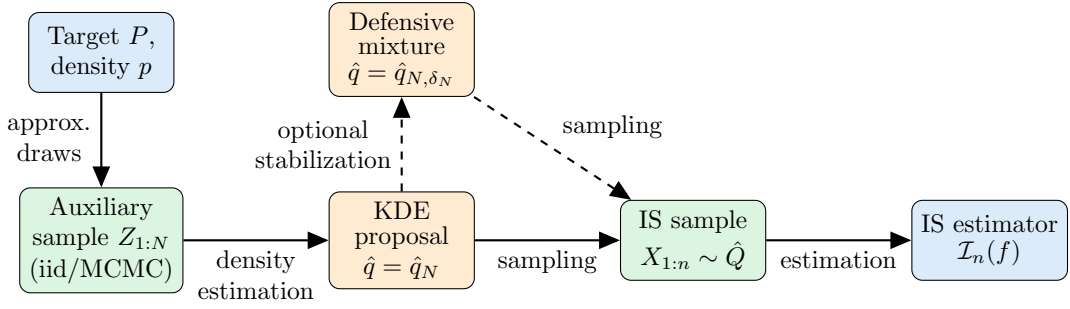

In this work we study an IS approach in which the proposal is \emph{learned from data}.
Specifically, we consider the two-stage pipeline illustrated in \Cref{fig:nis_pipeline}: an auxiliary sample $Z_1,\dots,Z_N$, approximately distributed according to $P$ (e.g.\ i.i.d.\ samples or generated by an MCMC algorithm targeting $P$), is used to construct a data-driven proposal $\hat Q$ with density $\hat q$, and subsequently $X_1,\dots,X_n\sim \hat Q$ are drawn to form an IS estimator.
A canonical choice is to take $\hat q$ to be a kernel density estimator (KDE) \cite{Silverman1986KDE,WandJones1995}.
Such two-stage constructions are common in Bayesian workflows \citep{Martino2017LAIS,papaioannou_improved_2019,uribe2021}, where pilot runs or previous simulations provide approximate samples from the target distribution.
A key feature of this setting is the presence of \emph{two interacting sample sizes}: 
\begin{itemize}
\item the \emph{proposal sample size} $N$, used to learn $\hat q$, and
\item the \emph{importance sampling size} $n$, used to estimate $\I(f)$.
\end{itemize}
This hierarchical structure leads to a two-sample regime in which the overall accuracy depends jointly on the proposal sample size $N$ and the importance sampling size $n$, as well as on how these quantities are balanced.
Understanding this interaction is crucial for both theory and practice. 

Despite its practical appeal, the theoretical properties of IS with learned proposal distributions remain poorly understood, especially in the presence of dependent auxiliary data \citep{Botev2013,delyon21,schuster21}.
Two challenges arise in the analysis of this pipeline.
First, the proposal $\hat Q$ is random, so classical IS theory for fixed proposals does not directly apply.
Second, KDE-based proposals can be inaccurate in the tails, which may lead to unstable importance weights.
To address this, we also consider a \emph{defensive} construction in which the learned proposal is mixed with a fixed reference density $\varphi$,
\begin{equation}
\label{equ:defensive_KDE_mixture_intro}
\hat q_{N,\delta_N}
\coloneqq
(1-\delta_N)\hat q_N + \delta_N \varphi,
\qquad \delta_N \in (0,1),
\end{equation}
ensuring stability of the importance weights while retaining flexibility \citep{Hesterberg1995Defensive,Owen2000safe}.

\paragraph*{Contributions}
The goal of this paper is to quantify how the error incurred when estimating the proposal distribution propagates into the error of the resulting IS estimator, and to derive explicit non-asymptotic guarantees that capture the interplay between $N$ and $n$. The main contributions are as follows.

\begin{enumerate}[label=(C\arabic*),leftmargin=3em]

\item
\label{item:IS_error_bounds_random_proposal}
\textbf{Non-asymptotic error bounds for IS with random proposals.}
We develop a general theory for IS with data-driven (random) proposals, including standard, defensive, and self-normalized defensive estimators.
Our bounds explicitly separate the Monte Carlo error, scaling as $n^{-1/2}$, from the proposal approximation error, quantified via 
the mean integrated absolute error ($\MIAE$) and mean integrated squared error ($\MISE$), providing a general mechanism for translating density estimation accuracy into IS accuracy in hierarchical settings.
In particular, we establish explicit variance and error bounds under transparent integrability conditions, quantify how the defensive weight $\delta_N$ trades off robustness and approximation accuracy, and characterize the resulting variance inflation in terms of the proposal quality.

\item
\label{item:MISE_MIAE_bounds_KDE_MCMC}
\textbf{New $\MIAE$ and $\MISE$ bounds for KDEs based on MCMC data.}
We derive bounds for the $\MIAE$ and $\MISE$ of KDEs constructed from both i.i.d.\ samples and Markov chains.
For geometrically ergodic chains we obtain explicit finite-sample rates, including new $\MIAE$ bounds and quantitative $\MISE$ bounds that capture the effect of dependence even in non-stationary regimes.
These results form a key ingredient in our IS analysis and are of independent interest in their own right.

\item \textbf{Explicit rates in the two-sample regime $(N,n)$.}
By combining our IS bounds from \ref{item:IS_error_bounds_random_proposal} with the KDE error estimates from \ref{item:MISE_MIAE_bounds_KDE_MCMC}, we obtain explicit convergence rates that depend jointly on $N$ and $n$.
These results make precise how the estimation error of the KDE-based (possibly defensive) proposal interacts with the Monte Carlo error, and provide guidance for choosing the sample sizes $(N,n)$, the bandwidth $h_N$, and the defensive weight $\delta_N$. We illustrate our theoretical findings on a simple one-dimensional example with a Laplace target distribution and a defensive KDE proposal using a Gaussian kernel and a Cauchy defensive component.
\end{enumerate}

\paragraph*{Outline}
After reviewing related work in \Cref{sec:RelatedWork}, we introduce the hierarchical sampling framework and notation in \Cref{sec:Notation}.
In \Cref{sec:bounds_is_estimators} we develop general non-asymptotic bounds for IS with random proposals, including defensive and self-normalized variants.
\Cref{sec:kde_mcmc_mise} derives $\MIAE$ and $\MISE$ bounds for KDEs based on i.i.d.\ and Markov chain data.
These results are combined in \Cref{sec:kde_is} to obtain explicit convergence rates in the two-sample regime $(N,n)$ and practical guidance for parameter selection.
\Cref{sec:numerics} illustrates the theoretical results on a simple numerical example.

\section{Related Work}
\label{sec:RelatedWork}
Importance sampling is a classical Monte Carlo technique for estimating expectations with respect to a target distribution by sampling from an alternative proposal distribution.
Its efficiency critically depends on the choice of the proposal, and a large body of literature studies how to construct proposals that reduce the variance of the resulting estimator.
Early theoretical analyses characterize the optimal proposal distribution that minimizes the estimator variance, but this optimal distribution typically depends on the unknown quantity being estimated and is therefore not directly available in practice.
For general references on IS we refer to, e.g.,~\cite{Agapiou2017IS,owen_monte_2013,Robert2004MC}.

A substantial line of work focuses on constructing proposal distributions by optimizing over parametric families. A common strategy is to minimize a divergence between the proposal and the optimal importance sampling distribution. In particular, minimizing the Kullback--Leibler divergence leads to convex optimization problems when the proposal family belongs to the exponential family. This idea forms the basis of the cross-entropy method~\cite{rubinstein_cross_2004}, which sequentially updates the parameters of a proposal distribution using samples generated from intermediate distributions. The cross-entropy framework has been widely studied, with numerous algorithmic refinements and theoretical analyses~\cite{beh2025insight,demange-chryst2024variational,el_masri2024,geyer_cross_2019,homem-de-mello_study_2007,ELMASRI2021,papaioannou_improved_2019,uribe2021}.

Another line of work constructs proposal distributions from samples generated by Markov chain Monte Carlo (MCMC), an approach commonly referred to as Markov chain importance sampling (MCIS)~\cite{Botev2013}. In MCIS, a Markov chain targeting a suitable distribution first generates an auxiliary sample from which a density estimator is fitted to approximate the target distribution, and this estimate is then used as a proposal for a subsequent importance sampling stage. A closely related algorithmic framework is layered adaptive importance sampling (LAIS) \citep{Martino2017LAIS}, which combines a population of MCMC chains with an importance sampling stage, resulting in a similar two-level sampling structure.

Various density estimators have been proposed for constructing proposal distributions in this context.
In the rare-event simulation literature, parametric approaches include Gaussian mixture models~\citep{ehre2024steinvariationalrareevent,wagner22} and von Mises--Fisher--Nakagami mixture models~\citep{beh2025affineinvariantinteractinglangevin,papaioannou_improved_2019,uribe2021}.
Low-rank mixture models based on probabilistic principal component analysis have recently been proposed for high-dimensional settings~\cite{kruse2025scalableimportancesamplinghigh}.
Semi-parametric constructions based on conditional density factorizations were introduced in~\cite{Botev2013}.
Alternative approaches fit parametric models whose parameters maximize the variance of the log-density, motivated by properties of the optimal importance sampling distribution~\cite{chan_improved_2012}. A variant that exploits rejected proposals in Metropolis--Hastings and related algorithms via importance sampling was recently proposed by~\cite{schuster21} and~\cite{Rudolf2020}. The authors of \cite{schillings2020convergence} study the use of the Laplace approximation as an importance sampling proposal and establish high-probability convergence of the self-normalized estimator in the small-noise regime of Bayesian inverse problems.

Another line of work studies nonparametric approaches to importance sampling in which the optimal proposal distribution for a given integrand is approximated from data.
For example, \cite{zhang96} propose estimating the optimal importance sampling density using a kernel density estimator.
Their approach first constructs an importance sampling based estimator of the optimal proposal distribution and then uses this estimate in a second importance sampling stage.
Related ideas have been studied in \cite{neddermeyer2009computationally}, where both self-normalized and unnormalized estimators are considered and the kernel density estimator is replaced by a linear blend frequency polygon estimator, and in \cite{givens1996local}, where a localized variant based on kernel density estimation is proposed.
In all these works, the samples used to estimate the proposal distribution are generated as part of the algorithm itself, and the resulting proposal is tailored to a specific integrand through the corresponding optimal importance sampling density.
Our framework also applies when the auxiliary sample targets the optimal importance sampling distribution, or more generally another reference distribution different from $P$; see \Cref{rem:alternative_reference_distribution}.

For instance, other approaches learn proposals that approximate the target distribution itself.
The authors in \cite{delyon21} suggest a policy-based method that sequentially constructs a sequence of proposal distributions using IS.
Their policy is parameterized as a defensive mixture between a flexible kernel density estimator based on previously generated particles and a heavy-tailed reference density that ensures stability.
New particles are generated from the current proposal and then used to update the policy, resulting in an adaptive sequence of proposal distributions.
For this sequence of density approximations the authors establish a central limit theorem.
In contrast, the present work considers a setting in which the auxiliary samples used to construct the proposal are simply available, regardless of how they were generated.
These samples may arise, for example, from observations, previous simulations, or preliminary Monte Carlo runs.
We therefore focus on analyzing how the error of the resulting density estimator affects the accuracy of the subsequent IS estimator.
While the safe-adaptive proposal constructions in~\cite{delyon21} are related in spirit to the nonparametric proposals considered here, the theoretical focus is different: their analysis studies the asymptotic behavior of the sequence of approximating densities, whereas the present work quantifies how proposal estimation error propagates into the error of the resulting IS estimator.

Despite these developments, theoretical results that explicitly account for the randomness of data-driven proposals and quantify the impact of proposal estimation error on the performance of the resulting IS estimator remain limited.
The present work addresses this gap by providing a non-asymptotic analysis of IS with learned proposal distributions.

\section{Setup and Notation}
\label{sec:Notation}

Throughout the paper, $(\Omega,\Sigma,\bP)$ denotes an underlying probability space.
Let $\cP(\R^d)$ denote the space of probability measures on $(\R^d,\cB(\R^d))$, equipped with the Borel $\sigma$-algebra induced by the topology of weak convergence, that is, the $\sigma$-algebra generated by the evaluation maps
$\mathrm{ev}_B\colon \mu\mapsto \mu(B)$, $B\in\cB(\R^d)$.
The target distribution is denoted by $P\in\cP(\R^d)$ and is assumed to admit a strictly positive Lebesgue density
$p\colon \R^d\to\R_{>0}$.
As proposal distribution we consider random probability measures $\hat{Q}$ on $\R^d$, equipped with its Borel $\sigma$-algebra $\cB(\R^d)$. 
More precisely, the random proposal is a $\cP(\R^d)$-valued random variable, which we assume to be measurable with respect to a sub-$\sigma$-algebra $\cF\subseteq\Sigma$, with (random) Lebesgue density $\hat q$.
Throughout, we assume that $P \ll \hat Q$ almost surely.
Since the proposal $\hat Q$ is random, our analysis repeatedly uses conditional expectations of the form
\[
\E \big[ g(X_{i}) \mid \hat Q \big]
=
\int g\, \d \hat Q.
\]
This identity is justified by the universal sampling construction stated in
\Cref{sec:universal_sampling}, and throughout the paper we assume samples from $\hat Q$ to be realized in this way.

Our goal is to approximate expectations of the form $\I(f)\coloneqq \E_{P}[f]$ for $f\in L^2(P)$ using IS with proposal $\hat Q$.
Conditionally on $\hat Q$, let $X_1,\dots,X_n$ be i.i.d.\ samples from $\hat Q$.
The corresponding IS estimator is
\begin{equation}
\label{equ:IS_estimator}
\I_n(f)
\coloneqq
\frac1n\sum_{i=1}^n \hat\rho(X_i)\,f(X_i),
\qquad
\text{where}
\quad
\hat\rho \coloneqq \frac{\rd P}{\rd \hat Q} = \frac{p}{\hat q}
\end{equation}
is the corresponding (random) Radon--Nikodym derivative.
In the applications of interest, the proposal density $\hat{q}$ is constructed from auxiliary samples $Z_1,\dots,Z_N\in\R^d$ that are approximately distributed according to $P$ (either i.i.d.\ or obtained from a Markov chain targeting $P$).
In particular, we focus on kernel density estimators of the form
\begin{equation}\label{eq:our_is_proposal_density}
\hat q_N
\coloneqq
\frac{1}{N}\sum_{k=1}^N K_h(Z_k,\cdot),
\qquad
K_h(x,y)
\coloneqq
\frac{1}{h^d} \, K \left(\frac{x-y}{h}\right),
\end{equation}
where $K \in L^{1}(\bR^{d})$ is a kernel and $h = h_N$ is the KDE bandwidth.
In both cases, $\cF = \sigma(Z_1,\dots,Z_N)$ denotes the $\sigma$-algebra generated by the auxiliary sample.
To quantify the quality of the learned proposal, we use the mean integrated absolute error ($\MIAE$) and mean integrated squared error ($\MISE$):
\[
\MIAE(\hat q_N)
\coloneqq
\E\int_{\R^d} \big|\hat q_N(x)-p(x)\big|\,\d x,
\qquad
\MISE(\hat q_N)
\coloneqq
\E\int_{\R^d} \big(\hat q_N(x)-p(x)\big)^2\,\d x.
\]
These quantities naturally control the discrepancy between $p$ and $\hat q_N$ and will enter our IS error bounds.
Moreover, as will be discussed in \Cref{rem:relation_kl_tv}, they can themselves be controlled by classical divergences such as the total variation and Kullback--Leibler distances, which commonly arise in density estimation, variational inference, and approximate Bayesian computation.

\begin{table}[!t]
\centering
\caption{Overview of the random proposal classes considered in the paper and the corresponding assumptions and results.
Here, $\hat q_0$ denotes an arbitrary proposal density, $\varphi$ a strictly positive defensive mixture component, and $\delta,\delta_N\in(0,1)$ the corresponding mixture weights.}
\label{tab:proposal_types}
\renewcommand{\arraystretch}{1.35}
\begin{tabular}{C{4.0cm}  C{4.5cm}  C{5.5cm}}
\toprule
\textbf{Random proposal  $\hat q$}
&
\textbf{General}
&
\textbf{$N$-dependent}
\\
\midrule

\textbf{Non-defensive}
&
\makecell[c]{
$\hat q$ arbitrary \\
\Cref{assump:basic_is}\\ \Cref{thm:basic_random_is}
}
&
\makecell[c]{
$\hat q=\hat q_N$ KDE defined in \eqref{eq:our_is_proposal_density}\\
\Cref{thm:is_kde_final_rates}\,\ref{item:thm_is_kde_final_rates_stationary}\,\ref{item:thm_is_kde_final_rates_stationary_nondefensive}
and
\ref{item:thm_is_kde_final_rates_nonstationary}\,\ref{item:thm_is_kde_final_rates_nonstationary_nondefensive}
}
\\
\midrule
\makecell[c]{
\textbf{Defensive mixture}
}
&
\makecell[c]{
$\hat{q} = \hat q_\delta=(1-\delta)\hat q_0+\delta\varphi$
\\
\Cref{assump:safe_is}\\
\Cref{thm:variance_bound_for_safe_is,thm:error_bound_for_clipped_safe_snis}
}
&
\makecell[c]{
$\hat{q} = \hat q_{N,\delta_N} = (1-\delta_N)\hat q_N + \delta_N \varphi$
\\
\Cref{assump:N_dependent_safe_is}\\ \Cref{cor:error_bound_for_N_dependent_safe_is_and_snis}
\\[1.5ex]
KDE specialization: $\hat q_N$ as in \eqref{eq:our_is_proposal_density}\\
\Cref{thm:is_kde_final_rates}\,\ref{item:thm_is_kde_final_rates_stationary}\,\ref{item:thm_is_kde_final_rates_stationary_defensive}
and
\ref{item:thm_is_kde_final_rates_nonstationary}\,\ref{item:thm_is_kde_final_rates_nonstationary_defensive}
}
\\
\bottomrule
\end{tabular}
\end{table}

To improve tail robustness, we also consider the defensive mixture proposals of the form \eqref{equ:defensive_KDE_mixture_intro}.
\Cref{tab:proposal_types} summarizes the proposal classes considered in this paper together with the corresponding assumptions and main results.
All upper bounds in these results are understood in the extended real sense and may take the value $+\infty$.

\begin{remark}[Alternative reference distributions]
\label{rem:alternative_reference_distribution}
Throughout the paper we formulate our results for the setting in which the auxiliary sample
$Z_1,\dots,Z_N$ approximately targets the distribution $P$ itself, so that the learned proposal
$\hat q_N$ is constructed as an approximation of the target density $p$, with approximation quality measured in terms of $\MIAE(\hat q_N)$ and $\MISE(\hat q_N)$.
However, the analysis extends directly to settings in which the auxiliary sample targets a different reference distribution $\tilde P$ with density $\tilde p$.
Indeed, defining $\tilde f(x) = fp/\tilde{p}$, we obtain
\[
\I(f)
=
\E_{P}[f]
=
\int_{\R^d} f(x)\, p(x)\,\d x
=
\int_{\R^d} f(x)\,\frac{p(x)}{\tilde p(x)}\,\tilde p(x)\,\d x
=
\E_{\tilde P}[\tilde f].
\]
Consequently, all results in the paper remain valid after replacing $(P,p,f)$ by $(\tilde P,\tilde p,\tilde f)$.
This includes, for instance, settings in which the auxiliary sample targets the optimal importance sampling distribution, which is known to minimize the variance of the corresponding IS estimator \cite{Robert2004MC}, or a smoothed version of it, as is common in rare event estimation \citep{beh2025affineinvariantinteractinglangevin,papaioannou_improved_2019,uribe2021,wagner22}. 
\end{remark}

\section{General bounds for IS with random proposals}
\label{sec:bounds_is_estimators}

In this section, we develop non-asymptotic error and variance bounds for IS when the proposal distribution is random.
We work throughout under the hierarchical framework introduced in \Cref{sec:Notation}, where a random proposal $\hat Q$ is first constructed from auxiliary data and the IS estimator is formed from samples drawn conditionally on $\hat Q$.

Our analysis separates two distinct sources of error:

\begin{itemize}
\item the \emph{Monte Carlo error} due to the finite importance sampling size $n$, and
\item the \emph{proposal approximation error} induced by replacing the target density $p$ with the learned density $\hat q$, which will be quantified in terms of $\MIAE(\hat q)$ and $\MISE(\hat q)$.
\end{itemize}

We proceed in three steps.
First, in \Cref{sec:setting_1a}, we analyze the unbiased IS estimator under minimal assumptions on the random proposal.
Second, \Cref{sec:setting_1b} introduces a defensive mixture construction that guarantees tail robustness and yields explicit variance bounds.
Finally, \Cref{sec:self_normalized_extension} discusses the practically important self-normalized estimator.

\begin{remark}[Relation to KL and total variation distances]
\label{rem:relation_kl_tv}
The bounds in \Cref{thm:basic_random_is,thm:variance_bound_for_safe_is,thm:error_bound_for_clipped_safe_snis} are expressed in terms of $\MIAE$ and $\MISE$, which can in turn be controlled by more classical discrepancy measures between the target density $p$ and the learned proposal $\hat q_N$, such as (expected) total variation and Kullback--Leibler distances:
\begin{align*}
\MIAE(\hat q_N)
&= 2\,\E\left[\operatorname{TV}(p,\hat q_N)\right]
\;\le\;
\sqrt{2\, \E\left[\KL(p\|\hat q_N)\right]},
\end{align*}
where we used Pinsker's inequality and Jensen's inequality.
Moreover, if $p$ and $\hat q_N$ are bounded and $\hat{S}_{N} \coloneqq \sup_{x\in\R^d} p(x) + \sup_{x\in\R^d} \hat q_N(x)$, then
\begin{align*}
\MISE(\hat q_N)
&\le
2 \, \E\left[ \hat{S}_{N}\, 
\operatorname{TV}(p,\hat q_N)\right]
\le
\E\left[ \hat{S}_{N}\, 
\sqrt{2 \, \KL(p\|\hat q_N)}\right].
\end{align*}
Consequently, the bounds derived in this section can be combined with methods that control divergences such as $\KL(p\|\hat q_N)$ or $\operatorname{TV}(p,\hat q_N)$, connecting our IS error bounds directly to objectives commonly used in modern density estimation and approximate inference.
Similar bounds can be obtained for the Hellinger distance, since it is equivalent to the total variation distance.
\end{remark}

\subsection{Importance sampling with random proposals}
\label{sec:setting_1a}

We begin with the basic hierarchical IS setting introduced in \Cref{sec:Notation}, which we summarize in the following assumption.

\begin{assumption}[Importance sampling with random proposal]
\label{assump:basic_is}
Let $P \in \cP(\R^d)$ be a probability distribution with strictly positive Lebesgue density
$p \colon \R^d \to \R_{>0}$.
Let $\hat Q$ be a $\cP(\R^d)$-valued random variable with (random) Lebesgue density $\hat q$, measurable with respect to a sub-$\sigma$-algebra $\cF \subseteq \Sigma$, such that $P \ll \hat Q$ almost surely.
Conditionally on $\hat Q$, let $X_1,\dots,X_n$ be conditionally independent with common distribution $\hat Q$.
Let $Z \sim P$ be an auxiliary random variable independent of $\cF$.
\end{assumption}

We analyze the properties of the corresponding IS estimator \eqref{equ:IS_estimator}.
While the estimator is conditionally unbiased, its accuracy depends on both the importance sampling size $n$ and the quality of the random proposal $\hat Q$.
The following theorem summarizes these properties.

\begin{theorem}[Basic properties of IS with a random proposal]
\label{thm:basic_random_is}
Let $f \in L^2(P)$. Under \Cref{assump:basic_is}, the estimator \eqref{equ:IS_estimator} satisfies:

\begin{enumerate}[label=(\alph*)]

\item \textbf{Unbiasedness: }
\label{item:basic_random_is_unbiased}
$\E[\I_n(f) | \hat Q] = \I(f)$ almost surely.
In particular, $\E[\I_n(f)] = \I(f)$.

\vspace{1ex}

\item \textbf{Variance bound: }
\label{item:basic_random_is_variance}
$\displaystyle
\Var[\I_n(f)]
\le
\frac{1}{n}
\left(
\Var[f(Z)]
+
\int_{\R^d}
\E\bigl[|\hat\rho(x)-1|\bigr] \,
f(x)^2 \, p(x)\,\mathrm{d}x
\right).
$

\vspace{1ex}

\item \textbf{Non-asymptotic error bounds:}
\label{item:basic_random_is_error}
\begin{align*}
\E\bigl[|\I_n(f)-\I(f)|\bigr]
&\le
\frac{\|f\|_{L^2(P)}}{\sqrt{n}}
\Big(1
+
(4 n \MIAE(\hat q))^{1/3}\Big)^{3/2}
=
\mathcal{O}\big(
n^{-1/2}
+
\MIAE(\hat q)^{1/2}
\big).
\end{align*}
\end{enumerate}
\end{theorem}

\begin{proof}
For \ref{item:basic_random_is_unbiased}, the conditional independence of the $X_i$ and the conditional sampling identity from \Cref{sec:universal_sampling} imply
\begin{equation}
\label{equ:basic_random_is_unbiased}
\E[\I_n(f) \mid \hat Q]
=
\E[\hat\rho(X_1) f(X_1) \mid \hat Q]
=
\int \hat\rho f \,\mathrm d\hat Q
=
\int f \,\mathrm dP
=
\I(f).
\end{equation}
Taking expectations yields the unconditional statement.

\medskip
For \ref{item:basic_random_is_variance}, let $X \sim \hat Q$ conditionally on $\hat Q$.
By the conditional sampling identity from \Cref{thm:universal_sampling},
\begin{align*}
\E[(\hat\rho(X)f(X))^2 \mid \hat Q]
&=
\int (\hat\rho f)^2 \,\mathrm dQ
=
\int \hat\rho f^2 \,\mathrm dP
\le
\int |\hat\rho-1| f^2 \,\mathrm dP
+
\int f^2 \,\mathrm dP.
\end{align*}
Since $\E[\hat\rho(X)f(X)\mid \hat Q] = \I(f) = \bE[f(Z)]$ by the same calculation as in \eqref{equ:basic_random_is_unbiased}, the law of total variance implies
\begin{align*}
\Var[\hat\rho(X)f(X)]
&\le
\int \E[|\hat\rho-1|] f^2 \,\mathrm dP
+
\bE[f(Z)^{2}] - \bE[f(Z)]^{2}
\\ &=
\int \E[|\hat\rho-1|] f^2 \,\mathrm dP
+
\Var[f(Z)].
\end{align*}
By conditional independence of the $X_i$, $\Var[\I_n(f)] = \Var[\hat\rho(X)f(X)]/n$, which yields the claim.

\medskip
For \ref{item:basic_random_is_error}, let $v>1$ be arbitrary.
Using a bound from the proof of \citep[Theorem~1.1]{chadia18} and Jensen's inequality,
\begin{align*}
\E\bigl[|\I_n(f)-\I(f)|\bigr]
&=
\E\Big[\E\bigl[|\I_n(f)-\I(f)| \mid \cF \bigr]\Big]
\\
&\le
\E\Big[
\|f\|_{L^2(P)} \sqrt{\frac{v}{n}}
+
2\|f\|_{L^2(P)} \sqrt{\P(\hat\rho(Z)>v \mid \cF)}
\Big]
\\
&\le
\|f\|_{L^2(P)}
\left(
\sqrt{\frac{v}{n}}
+
2\sqrt{\E[\P(\hat\rho(Z)>v \mid \cF)]}
\right).
\end{align*}
By the law of total expectation, Markov's inequality, Jensen's inequality, and the independence of $Z$ and $\cF$,
\begin{align*}
\E[\P(\hat\rho(Z)>v \mid \cF)]
&=
\P(\hat\rho(Z)>v)
\\
&\le
\P \big( |\hat q(Z)-p(Z)|>(1-v^{-1})p(Z) \big)
\\
&\le
(1-v^{-1})^{-1} \, \bE\bigg[ \frac{|\hat q(Z)-p(Z)|}{p(Z)}\bigg]
\\
&=
\frac{v}{v-1}\,\MIAE(\hat q),
\end{align*}
which proves the $v$-dependent bound
\begin{align*}
\E\bigl[|\I_n(f)-\I(f)|\bigr]
&\le
\|f\|_{L^2(P)}
\left(
\sqrt{\frac{v}{n}}
+
2\sqrt{\frac{v}{v-1}\,\MIAE(\hat q)}
\right).
\end{align*}
This bound is minimized by optimizing the function
\[
g(v)
\coloneqq
\frac{\sqrt{v}}{a}+b \sqrt{\frac{v}{v-1}},
\qquad
a \coloneqq \sqrt{n},
\qquad
b \coloneqq 2 \sqrt{\MIAE(\hat q)},
\]
over $v>1$, which yields the unique minimizer $v_{\ast} = 1 + (ab)^{2/3}$.
Plugging $v_{\ast}$ into $g$ then gives the minimum
\[
g(v_\ast)
=
(a^{-1}+a^{-1/3}b^{2/3})\sqrt{1+(ab)^{2/3}}
=
a^{-1}\big(1+(ab)^{2/3}\big)^{3/2},
\]
finalizing the proof.
\end{proof}

The estimate in \Cref{thm:basic_random_is} separates the two fundamental error sources.
The term $n^{-1/2}$ corresponds to the usual Monte Carlo error, while the second term reflects the discrepancy between $p$ and the learned proposal $\hat{q}$ through $\MIAE(\hat q)$.
In particular, even for large $n$, the overall accuracy is limited by the quality of the proposal estimator.

In the deterministic-proposal case, i.e.\ when $\hat Q = Q$ almost surely, $\MIAE(\hat q)$ reduces to the integrated absolute error $\IAE(q) = \int_{\R^d} |q-p| = 2\,\operatorname{TV}(P,Q)$, and the bound can be viewed as an $L^1$-based counterpart of truncation-type arguments \citep{chadia18}.
Unlike classical variance bounds based on the $\chi^2$-divergence, it only requires $L^1$-closeness of $q$ to $p$, at the price of a residual term of order $\IAE(q)^{1/2}$ instead of a purely $n^{-1/2}$ rate.

This limitation, namely, that in the random-proposal setting the $n^{-1/2}$ rate may fail to hold uniformly due to the proposal approximation error, is further supported by the following lower bound.
It shows that the Monte Carlo rate cannot hold uniformly over the class of target distributions considered. In particular, for any $n\in\N$ there exists a target distribution $P_n$ such that the IS estimator based on a KDE suffers a constant error. This result is non-uniform: the target distribution may depend on $n$. Consequently, \Cref{prop:lowerbound} should be interpreted as a worst-case statement showing that \Cref{thm:basic_random_is} cannot, in general, be sharpened to a uniform $n^{-1/2}$ rate over a broad class of target distributions.

\begin{proposition}[Lower bound]\label{prop:lowerbound}
Let $d = 1$, $f(x) = x$ for $x \in \bR$ and $N \in \bN$ be arbitrary. Then there exists a sequence of probability measures $(P_{n})_{n\in\N}$ with second moments $\|f\|_{L^2(P_n)}^2$ that are uniformly bounded in $n \in \bN$,
associated random measures $\hat q_N^{(n)}$ and a constant $C = C(N) > 0$ such that, for each $n\in\bN$,
$\I(f) \coloneqq \bE_{P_{n}}[f] = 1$ is independent of $n$ and
\[
\E\bigl[|\I_n(f)-\I(f)|\bigr]
\ge
C \quad \text{for all}\ n\in\bN.
\]
Here, $\I_{n}$ is the IS estimator \eqref{equ:IS_estimator} under the target measure $P = P_{n}$ and $\hat{q} = \hat q_N$ is the KDE proposal given by \eqref{eq:our_is_proposal_density} with Gaussian kernel
\begin{align*}
K_h(x,y)&:= \frac1{\sqrt{2\pi h^2}} \exp\Big(-\frac{(x-y)^2}{2h^2}\Big)\, ,
\\ 
h^2
&=
h_N^{2}
\le
\hat V(Z_1,\dots,Z_N)
\coloneqq
\frac{1}{N-1} \sum_{i=1}^{N} (Z_i-\bar Z_N)^2\,,
\end{align*}
where $Z_1,\dots,Z_N \sim P_{n}$ are independent and
$\bar Z_N \coloneqq \frac1N\sum_{i=1}^NZ_i$.
\end{proposition}

\begin{proof}
The proof is given in \Cref{sec:proofs}.
\end{proof}

\subsection{Defensive importance sampling with random proposals}
\label{sec:setting_1b}

We now consider a defensive (safe) version of IS that stabilizes the weights by mixing the learned proposal with a fixed reference density \citep{Owen2000safe}.

\begin{assumption}[Defensive mixture proposal]
\label{assump:safe_is}
In addition to \Cref{assump:basic_is}, assume that the proposal density $\hat{q}$ has the form
\[
\hat{q}
=
\hat q_\delta
\coloneqq
(1-\delta)\hat q_0 + \delta \varphi,
\qquad \delta \in (0,1),
\]
where $\hat q_0$ is a random probability density and $\varphi$ is a deterministic probability density that is strictly positive on $\R^d$.
Further, let the integrand $f \in L^2(P)$ and the probability density $p$ satisfy the integrability conditions
\begin{equation}
\label{eq:intconditions}
\Norm{\frac{f^{2} \sqrt{p}}{\varphi}}_{L^2(P)} < \infty,
\quad
\Norm{\frac{\sqrt{p}}{\varphi}}_{L^2(P)} < \infty,
\quad
\int |p-\varphi| f^2 < \infty,
\quad
\MISE(\hat{q}_{0})
<
\infty.
\end{equation}
In particular, these conditions ensure that
\begin{align}
\label{equ:definition_sigma_PQ}
\sigma_{\delta}^{2}
&\coloneqq
\frac{1}{\delta} \MISE(\hat{q}_{0})^{1/2} \Norm{\frac{\sqrt{p}}{\varphi}}_{L^2(P)}
+
\frac{\delta}{(1-\delta)} \ <\  \infty\, ,
\\
\label{equ:definition_sigma_PQ_f}
\sigma_{\delta}^{2}(f)
&\coloneqq
\Var[f(Z)]
+
\frac{1}{\delta} \MISE(\hat{q}_{0})^{1/2} \Norm{\frac{f^{2} \sqrt{p}}{\varphi}}_{L^2(P)}
+
\frac{\delta}{(1-\delta)} \int \abs{p-\varphi} \, f^{2}
\ <\  \infty\, .
\end{align}
\end{assumption}

\begin{remark}[Sufficient tail condition]
\label{rem:tailassumption}
The first two integrability conditions in \eqref{eq:intconditions} can be ensured for $f \in L^4(P)$ via a suitable tail relation between $\varphi$ and $p$:
Suppose there exist constants $c_u, c_l, c_{\mathrm{tail}} > 0$ and a compact set $C \subset \R^d$ such that
\[
\max(\sup p, \sup \varphi) \le c_u < \infty,
\qquad
\varphi \ge c_l > 0 \text{ on } C,
\qquad
\varphi \ge \sqrt{c_{\mathrm{tail}}\,p} \text{ on } \R^d \setminus C.
\]
Such a relation is, for example, satisfied if $p$ has exponentially decaying tails while $\varphi$ is chosen with heavier (polynomial) tails, for instance a multivariate Cauchy density.
Then, for both $g = f$ and $g \equiv 1$, we obtain
\[
\Norm{\frac{g^{2} \sqrt{p}}{\varphi}}_{L^2(P)}^{2}
\ = \ 
\int \frac{g^4\, p^2}{\varphi^{2}}
\ \leq\ 
\int_{C}  \frac{c_{u}\, g^4\, p}{c_l^2}
+
\int_{\R^d\setminus C} \frac{g^4\, p}{c_{\mathrm{tail}}\,p}
\ \leq\ 
\bigg( \frac{c_{u}}{c_l^2} + \frac{1}{c_{\mathrm{tail}}} \bigg) \, \|g\|_{L^4(P)}^4
\ < \ 
\infty.
\]
\end{remark}

\begin{theorem}[Variance and error bounds for defensive importance sampling]
\label{thm:variance_bound_for_safe_is}
Under \Cref{assump:safe_is}, $\I_n(f)$ in \eqref{equ:IS_estimator}, based on the proposal $\hat{q} = \hat q_\delta$, is an unbiased estimator of $\mathcal{I}(f)$ and satisfies
\[
\Var[\I_n(f)]
\le
\frac{\sigma_{\delta}^{2}(f)}{n}\, ,
\qquad
\E\bigl[|\I_n(f)-\I(f)|\bigr]
\le
\frac{\sigma_{\delta}(f)}{\sqrt{n}}\, ,
\]
where $\sigma_{\delta}(f)$ is defined by \eqref{equ:definition_sigma_PQ_f}.
\end{theorem}

\begin{proof}
Since
\begin{align*}
|\hat\rho_\delta-1|
&=
\abs{
\frac{p}{(1-\delta)\hat q_0+\delta\varphi}
-
\frac{p}{(1-\delta)p+\delta\varphi}
+
\frac{p}{(1-\delta)p+\delta\varphi}
- 1
}
\\
&\le
\frac{p \, (1-\delta) \, |p-\hat q_0|}{\delta \varphi (1-\delta) p }
+
\frac{\delta \absval{p-\varphi}}{(1-\delta) p}
\\
&=
\frac{|p-\hat q_0|}{\delta  \varphi}
+
\frac{\delta \absval{p-\varphi}}{(1-\delta) p},
\end{align*}
the Cauchy--Schwarz inequality and Jensen's inequality imply
\begin{align*}
\int \bE \big[ \abs{\hat{\rho_{\delta}} - 1} \big] \, f^{2} \, p     
&\leq
\int \bE\big[ \abs{p-\hat{q}_{0}} \big] \frac{f^{2} \, p}{\delta \, \varphi}
+
\frac{\delta}{(1-\delta)} \int \abs{p-\varphi} \, f^{2}
\\
&\leq
\frac{1}{\delta} \bigg( \int \bE\big[ \abs{p-\hat{q}_{0}} \big]^{2}\bigg)^{\!\! 1/2}
\bigg(\int \frac{f^{4} \, p^{2}}{\varphi^{2}} \bigg)^{\!\! 1/2}
+
\frac{\delta}{(1-\delta)} \int \abs{p-\varphi} \, f^{2}
\\
&\leq
\frac{1}{\delta} \MISE(\hat{q}_{0})^{1/2} \Norm{\frac{f^{2} \sqrt{p}}{\varphi}}_{L^2(P)}
+
\frac{\delta}{(1-\delta)} \int \abs{p-\varphi} \, f^{2}.
\end{align*}
The claim follows from \Cref{thm:basic_random_is}\ref{item:basic_random_is_variance}.
\end{proof}

\subsection{Defensive self-normalized importance sampling with random proposals}
\label{sec:self_normalized_extension}

In many applications the normalizing constant of the target distribution is unknown, so that the unbiased IS estimator \eqref{equ:IS_estimator} is not directly available. Instead of the target $p$ only an unnormalized version $\tilde{p} = c\, p$ with unknown constant $c > 0$ is accessible.
A common alternative is the self-normalized importance sampling (SNIS) estimator as well as its denominator-clipped version,
\begin{equation}
\label{eq:selfnormIS}
\mathcal I^{\mathrm{SN}}_{n}(f)
\coloneqq
\frac{\sum_{i=1}^{n} \frac{ \tilde p \,f}{\hat q}(X_i)}
{\sum_{i=1}^{n} \frac{\tilde p}{\hat q}(X_i)},
\qquad
I_n^{\mathrm{SN},\tau}(f)
\coloneqq
\frac{\sum_{i=1}^n \frac{\tilde p\,f}{\hat q}(X_i)}
{\max\big(\tau n\, , \, \sum_{i=1}^n \frac{\tilde p}{\hat q}(X_i)\big)},
\quad
0 < \tau < c.
\end{equation}
which replace the exact weights by normalized importance weights. While this removes the need to evaluate the normalizing constant, it introduces a bias and changes the variance behavior. In this subsection we extend the previous analysis to the self-normalized setting and identify how the randomness of the proposal affects the resulting error bounds.
As in the previous subsection, we work with the defensive proposal $\hat{q} = \hat q_{\delta}=(1-\delta)\hat q_{0}+\delta\varphi$ introduced in \Cref{assump:safe_is}.

\begin{theorem}[Error bounds for defensive self-normalized importance sampling]
	\label{thm:error_bound_for_clipped_safe_snis}
	Under \Cref{assump:safe_is}, the estimators $\mathcal I^{\mathrm{SN}}_{n}(f)$ and $\mathcal I^{\mathrm{SN},\tau}_{n}(f)$ in \eqref{eq:selfnormIS} based on the proposal $\hat{q} = \hat q_\delta$ are (possibly biased) estimators of $\mathcal{I}(f)$ satisfying
	\begin{align*}			
			\E\bigl[|\I_{n}^{\mathrm{SN}}(f)-\I(f)|\bigr]
			&\le
			\frac{1}{\sqrt{n}}
			\bigg(
			\sigma_{\delta}(f)
			+ 
			\sqrt{2} \, \sigma_{\delta}
			\Big(
			\norm{f}_{L^2(P)}
			+
			\bE \big[ \norm{f}_{L^4(\hat{Q})}^{4} \big]^{1/4} \sigma_{\delta}^{1/2}
			\Big)		
			\bigg)\, ,
			\\
			\label{equ:error_bound_for_clipped_safe_snis}
			\E\bigl[|\I_{n}^{\mathrm{SN},\tau}(f)-\I(f)|\bigr]
			&\le
			\frac{1}{\sqrt{n}}
			\left(
			\sigma_{\delta}(f)
			+			
			\frac{c^{2} \, \sigma_{\delta}}{\tau ( c - \tau )}
			\,
			\bigg( \absval{\I(f)} + \frac{\sigma_{\delta}(f)}{\sqrt{n}} \bigg)
			\right)\, ,
	\end{align*}
    where $\sigma_{\delta}$ and $\sigma_{\delta}(f)$ are defined by \eqref{equ:definition_sigma_PQ} and \eqref{equ:definition_sigma_PQ_f}.
\end{theorem}

\begin{proof}
Throughout the proof we work conditionally on $\cF := \sigma(\hat Q)$.
For notational convenience, we write
$\E^{\cF} \coloneqq \E[\cdot|\cF]$ and $\bV^{\cF} \coloneqq \bV[\cdot|\cF]$
and invoke \Cref{thm:universal_sampling} implicitly when passing to unconditional expectations.
First note that, since
\begin{align*}
\frac{1}{\hat q_\delta}
&\le
\left|
\frac{1}{(1-\delta)\hat q_0+\delta\varphi}
-
\frac{1}{(1-\delta)p+\delta\varphi}
\right|
+
\frac{1}{(1-\delta)p+\delta\varphi}
\le
\frac{(1-\delta)|p-\hat q_0|}{\delta \varphi (1-\delta) p}
+
\frac{1}{(1-\delta)p}.
\end{align*}
the Cauchy--Schwarz inequality and Jensen's inequality imply
\begin{equation}
\label{equ:thm_error_bound_for_clipped_safe_snis_first_bound}
\begin{split}
\bE\bigg[\int  \frac{p^{2}}{\hat{q}_{\delta}} \bigg]
&\leq
\frac{1}{\delta} \int \bE\big[ \abs{p-\hat{q}_{0}} \big] \frac{p}{\varphi}
\ +\ 
\frac{1}{(1-\delta)} \int p
\\
&\leq
\frac{1}{\delta} \bigg( \int \bE\big[ \abs{p-\hat{q}_{0}} \big]^{2} \bigg)^{1/2}
\bigg( \int \frac{p^{2}}{\varphi^{2}} \bigg)^{1/2}
+ 
\frac{1}{(1-\delta)}
\\
&\leq
\frac{1}{\delta} \MISE(\hat{q}_{0})^{1/2} \Norm{\frac{\sqrt{p}}{\varphi}}_{L^2(P)}
+
\frac{\delta}{(1-\delta)}
+ 1.
\end{split}
\end{equation}
Now define the i.i.d.\ random variables $W_{i} \coloneqq \frac{p}{\hat q} (X_i)$ as well as $\bar W_n \coloneqq n^{-1} \sum_{i=1}^n W_i$ such that
\[
\mathcal I_{n}(f)
=
\bar W_{n} \, \mathcal I_{n}^{\mathrm{SN}}(f),
\quad
\mathcal I_{n}(f)
=
\max(\tau_{c},\bar W_{n}) \, \mathcal I_{n}^{\mathrm{SN},\tau}(f),
\quad
\text{where}
\quad
\tau_{c} \coloneqq \frac{\tau}{c}\in (0,1).
\]
Since
\[
\bE^{\cF}[W_1]
=
\int \frac{p}{\hat q} \, \hat q
=
1,
\qquad
\bE^{\cF}[\bar W_n]
=
1,
\]
\eqref{equ:thm_error_bound_for_clipped_safe_snis_first_bound} implies
\begin{alignat*}{3}
\bV[\bar W_n]
\ &=\ 
\bE \big[ \bV^{\cF}[\bar W_n] \big] + \bV \big[ \bE^{\cF}[\bar W_n] \big]
&\ &=\ 
\frac{1}{n} \, \bE \big[ \bV^{\cF}[W_1] \big] + 0
\\
&=\ 
\frac{1}{n} \, \bE \bigg[ \int \frac{p^{2}}{\hat q^{2}} \, \hat q 
\ - \ 1
\bigg]
&\ &\leq\ 
\frac{1}{n} \, \bigg(
\frac{1}{\delta} \MISE(\hat{q}_{0})^{1/2} \Norm{\frac{\sqrt{p}}{\varphi}}_{L^2(P)}
+
\frac{\delta}{(1-\delta)}
\bigg)
\\
&=\ 
\frac{\sigma_{\delta}^{2}}{n}\, .
\end{alignat*}
For $D \coloneqq \{ \bar W_n \leq \tau_{c} \}$ and $E \coloneqq \{ \bar W_n \leq 1/2 \}$, Chebyshev's inequality and the Cauchy--Schwarz inequality give:
\begin{alignat*}{3}
    \bP(D)
    \ &\leq\ 
    \bP \big( \absval{\bar W_n - 1} \geq 1-\tau_{c} \big)
    &\ &\leq\ 
    \frac{ \bV[\bar W_n]}{(1-\tau_{c})^{2}}\, ,
    \\[2ex]
    \bP(E)
    \ &\leq\ 
    \bP \big( \absval{\bar W_n - 1} \geq 1/2 \big)
    &\ &\leq\ 4 \, \bV[\bar W_n] \, ,
    \\[2ex]
    \bE\big[ \I_{n}^{\mathrm{SN}}(f^2) \, \mathds{1}_{\Omega \setminus E}\big]
    \ &\leq\ 
    2\, \bE\big[ \I_{n}(f^2)\big]
    &\ &=\ 
    2 \, \bE_{P} \big[ f^2 \big],
    \\[2ex]
    \bE\big[ \I_{n}^{\mathrm{SN}}(f^2) \, \mathds{1}_{E}\big]
    \ &\leq\ 
    \bE\Big[ \max_{i=1,\dots,n}f^2(X_{i}) \, \mathds{1}_{E}\Big]
    &\ &\leq\ 
    \bE\Big[ \max_{i=1,\dots,n}f^4(X_{i})\Big]^{1/2} \, \bP(E)^{1/2}
    \\
    \ &\leq\ 
    2 \Big(\bE \Big[ n \, \bE^{\cF}\big[ f^4(X_{1}) \big] \Big] \, \bV[\bar W_n]  \Big)^{1/2}
    &\ &=\ 
    2 \, \bE \big[ \norm{f}_{L^4(\hat{Q})}^{4} \big]^{1/2} \, \sigma_{\delta}\, .
\end{alignat*}
Hence, the Cauchy--Schwarz inequality and Jensen's inequality yield
\begin{align*}
    \E\bigl[|\I_{n}(f)-\I_{n}^{\mathrm{SN}}(f)|\bigr]
    &=
    \E\bigl[|\I_{n}^{\mathrm{SN}}(f) \big(\bar{W}_{n} - 1\big)|\bigr]
    \\
    &\leq
    \E\bigl[|\I_{n}^{\mathrm{SN}}(f^2) | \bigr]^{1/2}
    \, \bV[\bar W_n]^{1/2}
    \\
    &\leq
    \Big(
    2 \, \bE_{P} \big[ f^2 \big]
    +
    2 \, \bE \big[ \norm{f}_{L^4(\hat{Q})}^{4} \big]^{1/2} \, \sigma_{\delta}
    \Big)^{\! 1/2}
    \, \bV[\bar W_n]^{1/2}
    \\
    &\leq
    \big(
    \norm{f}_{L^2(P)}
    +
    \bE \big[ \norm{f}_{L^4(\hat{Q})}^{4} \big]^{1/4} \sigma_{\delta}^{1/2}
    \big)
    \big( 2 \, \bV[\bar W_n]\big)^{1/2}.
    \intertext{Further, since
        $\absval{\max(\tau_{c},\bar{W}_{n}) - 1}
        \leq
        \absval{\bar{W}_{n} - 1} + \absval{\max(\tau_{c},\bar{W}_{n}) - \bar{W}_{n}}
        \leq
        \absval{\bar{W}_{n} - 1} + \tau_{c} \mathds{1}_{D}$,
        another application of the Cauchy--Schwarz inequality gives}
    \E\bigl[|\I_{n}(f)-\I_{n}^{\mathrm{SN},\tau}(f)|\bigr]
    &=
    \E\bigg[\absval{\I_{n}(f)} \abs{ \frac{\max(\tau_{c},\bar{W}_{n}) - 1}{\max(\tau_{c},\bar{W}_{n})}} \bigg]
    \\
    &\leq
    \tau_{c}^{-1}
    \E\big[\absval{\I_{n}(f)} \absval{ \bar{W}_{n} - 1 } \big]
    +
    \E\big[\absval{\I_{n}(f)} \, \mathds{1}_{D} \big]
    \\
    &\leq	
    \E\big[\I_{n}(f)^{2}\big]^{1/2} \,
    \Big(
    \tau_{c}^{-1}
    \bV[\bar{W}_{n}]^{1/2}
    +
    \bP(D)^{1/2}
    \Big)
    \\
    &\leq	
    \big( \Var[\I_n(f)] + \bE[\I_n(f)]^{2} \big)^{1/2} \, 
    \bV[\bar{W}_{n}]^{1/2}\, 
    \big(
    \tau_{c}^{-1}	
    +
    (1 - \tau_{c})^{-1}
    \big)
    \\
    &\leq	
    \Big( \Var[\I_n(f)]^{1/2} + \absval{\I(f)} \Big)\, \frac{c^{2}\, \bV[\bar{W}_{n}]^{1/2}}{\tau ( c - \tau )},
\end{align*}	
where we have used that $\I_n(f)$ is an unbiased estimator of $\I(f)$ (\Cref{thm:variance_bound_for_safe_is}).
The claim now follows from \Cref{thm:variance_bound_for_safe_is} and the triangle inequality (similarly for $\I_{n}^{\mathrm{SN},\tau}$)
\[
\E\bigl[|\I_{n}^{\mathrm{SN}}(f)-\I(f)|\bigr]
\le
\E\bigl[|\I_{n}^{\mathrm{SN}}(f)-\I_{n}(f)|\bigr] + \E\bigl[|\I_{n}(f)-\I(f)|\bigr]. 
\]	
\end{proof}

\begin{remark}[Clipping, implementability, and tail behavior]
	\label{rem:clipping_combined}
	
	The error bound for the standard SNIS estimator in
	\Cref{thm:error_bound_for_clipped_safe_snis} involves the term
	$\E[\|f\|_{L^4(\hat Q)}]$, which is undesirable because it imposes
	additional integrability requirements on the random proposal that are
	particularly restrictive for defensive proposals. This dependence is,
	however, largely unavoidable without clipping: controlling the event of a
	small random denominator $\bar W_n$ typically requires moment conditions
	strong enough to bound $\E[1/\bar W_n]$, and such conditions generally
	fail in defensive settings. For instance, with a Gaussian target and a
	heavy-tailed defense (e.g.\ Cauchy $\varphi$), the necessary higher-order
	moment assumptions need not hold. Denominator clipping provides a
	principled way to circumvent this difficulty while preserving the optimal
	$n^{-1/2}$ rate.
	
	\begin{enumerate}[label=(\roman*)]
		
		\item \textbf{Implementability and scale-free clipping.}
		The condition $0<\tau<c$ in \eqref{eq:selfnormIS} formally involves the
		unknown normalizing constant $c$. In practice this requirement should be
		interpreted as a scaling guideline. One may substitute a pilot estimator
		$\hat c$ of $c$ and, for instance, choose
		$\tau
		\coloneqq
		\tau_{0} \, 
		\hat{c}$ for some $\tau_{0} \in (0,1)$.		
				
		\item \textbf{Asymptotic inactivity of clipping.}
		If $\tau<c$ (equivalently $\tau_c \coloneqq \tau/c < 1$), then $\bar W_n\to 1$ in probability
		and hence $\P(\bar W_n\le \tau_c)\to 0$.
		If the importance weights are uniformly bounded (e.g.\ when
		$\tilde p/\varphi$ is bounded, as is typical in the defensive setting),
		Hoeffding's inequality even guarantees exponential decay.
		Thus, denominator clipping is asymptotically inactive and the clipped and
		standard SNIS estimators are asymptotically equivalent under
		\Cref{assump:safe_is}.
		
		\item \textbf{Relation to weight clipping.}
		An alternative stabilization replaces $W_i$ by $W_i\wedge M$ before
		normalization. While this controls large-weight variability, denominator
		clipping instead targets the complementary failure mode by preventing
		instability caused by unusually small empirical normalizing constants.
		Moreover, weight clipping introduces systematic bias and breaks the exact
		importance sampling identity, whereas denominator clipping leaves the
		weights unchanged and only modifies the estimator on the rare event
		$\{\bar W_n\le\tau_c\}$, which makes it particularly convenient for
		non-asymptotic analysis.
		
	\end{enumerate}
\end{remark}

\subsection{Consequences for $N$-dependent proposals}

The framework developed in the previous subsections applies to general random proposal distributions.
In the applications considered in this paper, however, the proposal is typically constructed from an auxiliary sample of size $N$, for instance through kernel density estimation, leading to a two-sample regime in which the overall accuracy depends jointly on the proposal sample size $N$ and the importance sampling size $n$.
The purpose of this subsection is to make the dependence on $N$ explicit and to quantify how convergence rates for the proposal estimator, expressed in terms of $\MISE(\hat q_N)$ as $N \to \infty$, translate into convergence rates for the resulting IS and SNIS estimators.
The following assumption formalizes this setting.
\begin{assumption}[$N$-dependent defensive mixture proposal]
\label{assump:N_dependent_safe_is}
In addition to \Cref{assump:basic_is,assump:safe_is}, assume that both $\hat{q}_{0} = \hat q_N$ and $\delta = \delta_{N} \in (0,1)$ may depend on a parameter $N \in \bN$.
In other words, the defensive mixture proposal takes the form \[
\hat{q}
=
\hat q_{N,\delta_N}
\coloneqq
(1-\delta_N)\hat q_N + \delta_N \varphi.
\]
\end{assumption}

\begin{corollary}
\label{cor:error_bound_for_N_dependent_safe_is_and_snis}
Under \Cref{assump:N_dependent_safe_is}, if $\MISE(\hat q_N) = \mathcal{O}(N^{-2 \eta})$ and $\delta_N \asymp N^{-\eta/2}$, then each of the estimators $\I_n^{\ast} \in \{ \I_n,\,  \I^{\mathrm{SN}}_{n},\, \I^{\mathrm{SN},\tau}_{n} \}$ defined in \eqref{equ:IS_estimator} and \eqref{eq:selfnormIS} satisfies
\begin{equation*}
\E\bigl[|\I_n^{\ast}(f)-\I(f)|\bigr]
\leq
\frac{1}{\sqrt{n}} \left(\Var[f(Z)]^{1/2} + \mathcal{O}\big( N^{-\eta/4} \big)\right),
\end{equation*}
where for $\I_n^{\ast} = \I^{\mathrm{SN}}_{n}$ we additionally require $\bE\big[ \norm{f}_{L^4(\hat{Q})}^{4} \big]$ to be finite.
\end{corollary}

\begin{proof}
The claim follows directly from \Cref{thm:variance_bound_for_safe_is} and \Cref{thm:error_bound_for_clipped_safe_snis} since
\[
\sigma_{\delta}^{2}
=
\mathcal{O}\big( N^{-\eta/2} \big),
\qquad
\sigma_{\delta}^{2}(f)
=
\Var[f(Z)] + \mathcal{O}\big( N^{-\eta/2} \big).
\]
\end{proof}

\begin{remark}[Interpretation of \Cref{cor:error_bound_for_N_dependent_safe_is_and_snis}]
The bounds in \Cref{cor:error_bound_for_N_dependent_safe_is_and_snis} show that IS with an $N$-dependent defensive proposal (e.g.\ defensive KDE proposal) achieves essentially the same $\mathcal{O}(n^{-1/2})$ error rate as direct sampling from $P$ if $\MISE(\hat q_N) \to 0$ as $N \to \infty$.
The use of an approximate proposal manifests as a variance inflation term, which vanishes as the proposal sample size $N$ increases.
Consequently, for large $N$, the estimator behaves nearly as in the ideal Monte Carlo setting.
\end{remark}

\section{MIAE and MISE bounds for KDE with i.i.d.\ and Markov chain input}
\label{sec:kde_mcmc_mise}

In this section, we analyze the approximation quality of kernel density estimates (KDEs)
\begin{equation}
\label{eq:kde_def_general}
\hat q_N
\coloneqq
\frac{1}{N}\sum_{k=1}^N K_h(Z_k,\cdot),
\qquad
K_h(x,y)
\coloneqq
\frac{1}{h^d}  \, K \left(\frac{x-y}{h}\right),
\end{equation}
constructed from auxiliary samples $Z_1,\dots,Z_N \in \R^d$ drawn either
independently from $P$ or generated by a Markov chain with invariant
distribution $P$. Our goal is to derive bounds on the mean integrated
squared error (MISE) and the mean integrated absolute error (MIAE),
which quantify the accuracy of the resulting proposal densities in the
IS framework developed in the previous section.
Throughout this section, we impose the following regularity conditions
on the kernel $K$, bandwidth $h=h_N$, and target density $p$.

\begin{assumption}[Kernel, bandwidth and smoothness assumptions]
\label{ass:kde_rates}
Let $K\colon\R^d\to\R_{\geq 0}$ be a nonnegative function, let $h=h_N>0$ be a scalar
bandwidth depending on the number $N$ of auxiliary samples, and let $p$
denote the Lebesgue density of $P\in\cP(\R^d)$.
\begin{enumerate}[label=(\roman*)]

\item $K \in L^1(\R^d)\cap L^2(\R^d)$ satisfies
\[
\int_{\R^d} K(u)\,\rd u = 1,
\qquad
\int_{\R^d} u\,K(u)\,\rd u = 0,
\qquad
\int_{\R^d} \|u\|^2 \, |K(u)|\,\rd u < \infty.
\]

\item The bandwidth satisfies $h_N \to 0$ and $N h_N^d \to \infty$ as
$N\to\infty$.

\item $p \in W^{2,1}(\R^d) \cap W^{2,2}(\R^d)$, i.e.\ $p$ and its weak derivatives up to order $2$ are both integrable and square-integrable.
\end{enumerate}
\end{assumption}
In particular, by the Cauchy--Schwarz inequality, \Cref{ass:kde_rates} implies that $K \in L^1(P)$ and thereby $K_h(x,\cdot)\in L^1(P)$ for all $x\in\R^d$ and $h>0$. Both the MISE and MIAE of $\hat q_N$ admit a decomposition into a
stochastic fluctuation term and a deterministic smoothing bias term
(for the MISE this is the classical bias--variance decomposition, while for the MIAE it follows from the triangle inequality):
\begin{align}
\label{eq:MISE_decomposition}
\MISE(\hat q_N)
&=
\underbrace{
\int_{\R^d}
\bE\bigl[(\hat q_N(x) - \bE_P[K_h(x,\cdot)])^2\bigr]\rd x
}_{\text{integrated variance}}
+
\underbrace{
\int_{\R^d}
\bigl(\bE_P[K_h(x,\cdot)] - p(x)\bigr)^2\rd x
}_{\text{integrated squared bias}}.
\\
\label{eq:MIAE_decomposition}
\MIAE(\hat q_N)
&\le
\underbrace{
\int_{\R^d}
\bE\bigl[|\hat q_N(x)-\bE_P[K_h(x,\cdot)]|\bigr]\rd x
}_{\text{integrated mean absolute deviation}}
+
\underbrace{
\int_{\R^d}
|\bE_P[K_h(x,\cdot)]-p(x)|\rd x
}_{\text{integrated absolute bias}}.
\end{align}
The bias terms in \eqref{eq:MISE_decomposition}--\eqref{eq:MIAE_decomposition}
depend only on the smoothing properties of the kernel and are independent
of the sampling scheme used to generate the auxiliary sample. In contrast,
the stochastic terms depend on whether the $Z_k$ are i.i.d.\ or generated by a Markov chain.
In the following subsections, we derive bounds for
these stochastic terms in both settings and combine them with the bias
bounds from \Cref{lemma:KDE_bias} to obtain rates for the MISE and MIAE.

In the Markovian setting, we distinguish two regimes. In \Cref{sec:MC_KDE_stationary}, we consider a
stationary chain satisfying a spectral gap assumption (cf.\ \Cref{ass:spectral_gap}).
As expected, in this setting we obtain essentially the same stochastic KDE rates as in the i.i.d.\ case, up to constants depending on the mixing properties of the chain.
In \Cref{sec:MC_KDE_non_stationary}, we turn to a non-stationary setting under the weaker drift and minorization conditions of \Cref{ass:Markov_chain}, which yield fully non-asymptotic bounds for geometrically ergodic chains.
The resulting non-stationary bounds are more conservative in their dependence on the bandwidth: Rather than the classical i.i.d.\ rate $1/(N h^d)$, our variance estimate scales like $1/(N h^{2d})$, which is likely not optimal.
This stronger bandwidth dependence comes from the proof technique used here: our argument is based on the non-asymptotic bounds of \cite{Latuszynski2013}, which are formulated in terms of the $V^{1/2}$-weighted supremum norm (here $V$ denotes a drift function, cf.\ \Cref{ass:Markov_chain}).
Whether one can recover the sharper $1/(N h^d)$ scaling under this weaker assumption appears to be open.

\begin{remark}
For i.i.d.\ input, many parts of the results in this section are classical \citep{Tsybakov2009,scott2011multivariate,WandJones1995}.
In particular, the bias terms in \eqref{eq:MISE_decomposition}--\eqref{eq:MIAE_decomposition}, the integrated variance bound, and the resulting $\MISE$ rate belong to the standard theory of kernel density estimation.
By contrast, we are not aware of a reference that states the corresponding $\MIAE$ bound under assumptions directly comparable to ours, namely the polynomial tail condition \eqref{eq:poly_tail_p_and_weighted_K2_moment}.
While truncation arguments can be used to extend $L^1$-consistency results from compactly supported densities to general densities \cite{DevroyeGyorgy1985} they do not directly yield the quantitative $\MIAE$ rates considered here without additional control of the tails of $p$.
For this reason, we impose the polynomial tail condition \eqref{eq:poly_tail_p_and_weighted_K2_moment} when deriving $\MIAE$ bounds.
We provide proofs of all stated results for completeness. 
\end{remark}
Before moving to the stochastic terms, we provide bounds for the bias terms in \eqref{eq:MISE_decomposition}--\eqref{eq:MIAE_decomposition}.
\begin{proposition}[KDE bias bound]
	\label{lemma:KDE_bias}
	Suppose \Cref{ass:kde_rates} holds and let $K_h$ be given by \eqref{eq:kde_def_general}.
    Then the associated smoothing bias satisfies, for certain constants $C_{\mathrm{bias},1},C_{\mathrm{bias},2} > 0$ that depend only on \(K\) and the corresponding Sobolev norms of \(p\),
	\begin{align*}
	\int_{\R^d}
	\left|
	\bE_P[K_h(x,\cdot)] - p(x)
	\right| \, \rd x
	&\le
	C_{\mathrm{bias},1}\, h^2 \, ,
	\\
	\int_{\R^d}
	\left|
	\bE_P[K_h(x,\cdot)] - p(x)
	\right|^2 \, \rd x
	&\le
	C_{\mathrm{bias},2}\, h^4 \, .
	\end{align*}
\end{proposition}

\begin{proof}
The proof is classical and given in \Cref{sec:proofs} for completeness.
\end{proof}

\subsection{KDE with i.i.d.\ data and stationary Markov chains with spectral gap}
\label{sec:MC_KDE_stationary}
For the stochastic error terms in \eqref{eq:MISE_decomposition}--\eqref{eq:MIAE_decomposition}, we begin with the more favorable setting in which the chain is stationary and
admits a spectral gap. In this regime one expects the dependence to affect the
KDE only through constants, so that the stochastic error has the same order as in the i.i.d.\ case, which we include in our results for completeness.
The following result makes this precise at the level of integrated variance and integrated mean absolute deviation bounds. 

KDE for dependent observations has been studied under a variety of dependence assumptions. Many results establish consistency and convergence rates for KDEs based on stationary mixing processes, often recovering the classical i.i.d.\ rates up to constants depending on the dependence structure \citep{robinson1983nonparametric,masry1986recursive,hart1990data,Yu1993,HS2017}.
The closest work to our considered stationary setting is \citep{KIM2026108271}, where the authors study KDE for data coming from an MCMC method. In a one-dimensional setting they derive similar error bounds on the MISE for Harris ergodic Markov chains matching the i.i.d.~data setting.

\begin{assumption}[Stationary Markov chain with spectral gap]
	\label{ass:spectral_gap}    
    $Z_1,\dots,Z_N$ are
    i.i.d. samples from $P$, or generated by a stationary Markov chain $(Z_k)_{k\in\N}$ in $\R^d$ with invariant distribution $P$ and a spectral gap $\gamma > 0$ or a pseudo spectral gap $\gamma_{\mathrm{ps}}>0$.
\end{assumption}

\begin{theorem}[Stochastic KDE error for i.i.d.\ and stationary Markov input]
\label{thm:MarkovKDE_stationary}
Suppose that \Cref{ass:kde_rates,ass:spectral_gap} hold and that $\hat{q}_{N}$ is the KDE given by \eqref{eq:kde_def_general}.
Then there exists a constant $C_{\mathrm{mc}} > 0$ such that
\begin{equation}
\label{equ:thm_MarkovKDE_stationary_L2_bound}
\int_{\R^d}
\E\bigl[
\bigl(\hat q_N(x)-\bE_P[K_h(x,\cdot)]\bigr)^{2}
\bigr]\rd x
\leq
\frac{C_{\mathrm{mc}}}{N h^d}\,\|K\|_{L^2(\R^d)}^2
=
\mathcal O\bigg(\frac1{N h_{N}^d}\bigg).
\end{equation}
Moreover, if the kernel $K$ and Lebesgue density $p$ of $P$ satisfy
\begin{equation}
	\label{eq:poly_tail_p_and_weighted_K2_moment}
	M_{K,m}
	\coloneqq
	\int_{\R^d} (1+\|u\|)^m K(u)^2\,\rd u
	<\infty,
	\qquad
	p(x)\le C_p (1+\|x\|)^{-m},
	\quad
	x\in\R^d,
\end{equation}
for certain constants $C_p>0$ and $m>2d$, then $C_{1} \coloneqq \sqrt{C_{\mathrm{mc}}\,C_p\,M_{K,m}} \int_{\R^d}(1+\|x\|)^{-m/2}\,\rd x < \infty$ and, for $0 < h < 1$,
\begin{equation}
\label{equ:thm_MarkovKDE_stationary_L1_bound}
\int_{\R^d}
\E\bigl[
\bigl|\hat q_N(x)-\bE_P[K_h(x,\cdot)]\bigr|
\bigr]\rd x
\leq
\frac{C_{1}}{\sqrt{N}\,h^{d/2}}
=
\mathcal O \bigg(\frac1{\sqrt{N}\,h_{N}^{d/2}}\bigg).
\end{equation}
\end{theorem}
\begin{proof}
The proof is given in \Cref{sec:proofs}.
\end{proof}

With \Cref{thm:MarkovKDE_stationary} in hand, the following classical KDE error bounds follow immediately.
\begin{corollary}[KDE error bounds based on stationary Markov chains with spectral gap]
\label{cor:kde_rates_stationary}
Under the assumptions of \Cref{thm:MarkovKDE_stationary}, and assuming in
addition \eqref{eq:poly_tail_p_and_weighted_K2_moment} for the $\MIAE$ bounds,
\begin{align}
   \label{MIAE_and_MISE_iid}
    &\MIAE(\hat q_N) = \mathcal O\Big(\frac{1}{\sqrt{Nh^d}} + h^2\Big)&
    &\text{and}&
    &\MISE(\hat q_N) = \mathcal O\Big(\frac{1}{Nh^d} + h^4\Big).&
\intertext{
Balancing the bias and variance terms yields the asymptotically optimal bandwidth
$h_N \asymp N^{-1/(d+4)}$,
and the corresponding optimal rates}
    \label{MIAE_and_MISE_iid_optimal_h}
    &\MIAE(\hat q_N) = \mathcal O \Big( N^{-2/(d+4)} \Big)&
    &\text{and}&
    &\MISE(\hat q_N) = \mathcal O \Big( N^{-4/(d+4)} \Big).& 
\end{align}
\end{corollary}

\begin{proof}
This follows directly from \Cref{lemma:KDE_bias} and \Cref{thm:MarkovKDE_stationary}.
\end{proof}

\subsection{KDE with non-stationary Markov chain input}
\label{sec:MC_KDE_non_stationary}

We now turn to the non-stationary setting.
While the behavior of the MISE for kernel density estimation in the stationary regime under mixing conditions is well understood, to the best of our knowledge, the results derived in this section are new and of independent interest.
Our analysis is based on the
non-asymptotic bounds for additive functionals of Markov chains established in
\cite{Latuszynski2013}.
To control the effect of dependence in the KDE, we work
in the geometrically ergodic regime, assuming a small set condition and a
drift condition.
These standard assumptions provide quantitative control of the Markov chain's mixing behavior and enable finite-sample error bounds \citep{Meyn2009MC}, without requiring stationarity or explicit spectral gap conditions. 

For a transition kernel $T \colon \mathbb R^d \times \mathcal B(\mathbb R^d) \to [0,1]$ and a measurable function $g \colon \mathbb R^d \to \mathbb R$, we write
$T[g](x) \coloneqq \int_{\mathbb R^d} g(y)\, T(x,dy), \ x\in\mathbb R^d,$
whenever the integral is well defined.
Iterates are denoted recursively by
$T^0[g]=g$ and $T^k[g]=T[T^{k-1}[g]]$ for $k\ge 1$.

\begin{assumption}[Markov chain and kernel conditions]
\label{ass:Markov_chain}
$Z_1,\dots,Z_N$ are generated by a Markov chain
$(Z_k)_{k\in\N}$ in $\R^d$ with $Z_{1} \sim \mu_1\in\cP(\R^d)$, time-homogeneous transition kernel $T\colon \R^d\times \cB(\R^d)\to [0,1]$ and invariant distribution $P$
such that $T$ is $P$-irreducible, aperiodic and Harris recurrent.
We impose the following conditions on the transition kernel $T$ and the KDE kernel $K$:
\begin{itemize}[leftmargin=4em, labelsep=0.3em]
\item[\textbf{(SSC)}]
\label{item:ass_Markov_chain_SSC}
\emph{Small set condition.}
There exist a set $J \in \cB(\R^d)$ with $\bE_{P}[\mathds{1}_J] > 0$, a constant $\beta > 0$, and a probability measure $\nu \in \cP(\R^d)$ such that for all $x \in J$,
\[
T(x, A) \ge \beta \nu(A)
\quad \text{for all Borel sets } A \subseteq \R^d.
\]
\item[\textbf{(DC)}]
\label{item:ass_Markov_chain_DC}
\emph{Drift condition.}
There exist a function $V: \R^d \to [1,\infty)$ and constants $\lambda \in (0,1)$ and $b > 0$ such that for all $x \in \R^d$,
\[
T[V](x) \le \lambda V(x)\mathds{1}_{\{x \notin J\}}
+ b\,\mathds{1}_{\{x \in J\}}\,,
\]
where $J$ is the set from (SSC). Moreover, assume that $\E_{\mu_1}[V]\le \frac{b}{1-\lambda}$.\footnote{We note that $\E_{\mu_1}[V]\le \frac{b}{1-\lambda}$ is required to uniformly bound $\E_{\mu_1}[T^N[V]]$ using \cite[Proposition 4.5]{Latuszynski2013}. For instance, it can be guaranteed when $\mu_1$ is a Dirac measure.}

\item[\textbf{(KMC)}]
\label{item:ass_Markov_chain_KMC}
\emph{Kernel moment condition.}
Assume that $K\in L^4(\R^d)$, and for all $r\in\{1,2,4\}$ there exists a constant $B_r = B_{r}(V,K)>0$, such that, for every $0 < h \leq 1$,
\[
\int_{\R^d}
\sup_{y \in \R^d}
|K_h(x,y)|^r V^{-r/2}(y)\,\rd x
\le \frac{B_r}{h^{rd}}\,.
\]
\end{itemize}
\end{assumption}
While \Cref{ass:Markov_chain} (SSC) and (DC) are classical conditions in the Markov chain literature \citep{Meyn2009MC,Latuszynski2013}, the kernel moment condition (KMC) is less common. It controls how the localization and tail decay of the KDE kernel interact with the growth of the drift function $V$.
The following lemma shows that (KMC) holds whenever 
$V^{-1/2}$ and $K$ admit non-increasing radial envelopes that decay at least
polynomially of order $r>d$.
In particular, the assumption is satisfied for commonly used kernels such as the Gaussian kernel, and for weight functions of the form
\[
V(x)=\exp(\kappa \|x\|) \quad \text{for some }\kappa>0,
\qquad\text{or}\qquad
V(x)=1+\|x\|^m \quad \text{for some } m>2d.
\]

\begin{lemma}\label{lem:weightedVnorm}
	Let $r\ge 1$,  $0 < h \leq 1$, and let $V\colon \R^d\to [1,\infty)$ and $K\colon \R^d\to[0,\infty)$ be measurable functions.
    Assume that there exist non-increasing functions $\Phi\colon [0,\infty) \to [0,1]$ and $\Psi \colon [0,\infty) \to  [0,\infty)$ such that, for every $x \in \bR^{d}$,
	\begin{align*}
	V^{-1/2}(x)&\le \Phi(\|x\|),&
	M_{\Phi,r}^{(d)}
	&\coloneqq
	|\mathbb S^{d-1}| \int_0^\infty \Phi(t)^r t^{d-1}\, \rd t
	<
	\infty,
	\\
	K(x) &\le \Psi(\|x\|),&
	M_{\Psi,r}^{(d)}
	&\coloneqq
	|\mathbb S^{d-1}| \int_0^\infty \Psi(t)^r t^{d-1}\, \rd t
	<
	\infty,
	\end{align*}	
    where $|\mathbb S^{d-1}|$ denotes the surface measure of the $(d-1)$-dimensional unit sphere in $\mathbb R^d$.
	Then
	\[
	S_h(x)
	\coloneqq
	\frac1{h^d}\sup_{y\in\R^d}
	K\left(\frac{x-y}{h}\right)V^{-1/2}(y),
	\qquad x\in\R^d,
	\]
	satisfies
	\[
	\int_{\R^d} S_h(x)^r\,\rd x
	\le
	\frac{2^d \, \Psi(0)^r \, M_{\Phi,r}^{(d)}}{h^{rd}}
	+
	\frac{2^d \, M_{\Psi,r}^{(d)}}{h^{(r-1)d}}\, .
	\]
	In particular, \Cref{ass:Markov_chain} (KMC) is satisfied with 
	$B_r =
	2^d\Big(
	\Psi(0)^r  \, M_{\Phi,r}^{(d)}
	+
	 M_{\Psi,r}^{(d)}
	\Big)$.
\end{lemma}

\begin{proof}
The proof is given in \Cref{sec:proofs}.
\end{proof}

\begin{lemma}
\label{lemma:Latuszynski_summary}
Under \Cref{ass:Markov_chain}, there exist constants $\alpha_{0},\alpha_{1} > 0$ such that, for each $N \in N$ and $g \in L^{1}(P)$,
\begin{equation*}
\E\bigg[\Big|\frac1N\sum_{k=1}^{N} g(Z_k) - \bE_{P}[g]\Big|\bigg]
\le
\E\bigg[\Big(\frac1N\sum_{k=1}^{N} g(Z_k) - \bE_{P}[g]\Big)^2\bigg]^{\frac{1}{2}}
\le
\frac{\alpha_{0}\, \|\bar{g}\|_{V^{1/2}}}{\sqrt{N}}
+
\frac{\alpha_{1}\, \|\bar{g}\|_{V^{1/2}}^2}{N}\,,
\end{equation*}
where
\begin{equation*}
\bar{g} \coloneqq g - \bE_{P}[g],
\qquad
\|\bar{g}\|_{V^{1/2}}^2 \coloneqq \sup_{y\in\R^d}\ |\bar{g}(y)|^2 V^{-1}(y)\,.
\end{equation*}
\end{lemma}

\begin{proof}
The first inequality is a direct application of Jensen's inequality, while the second one follows from \citep[Theorems~3.1 and~4.2]{Latuszynski2013}, after noting that by \citep[Proposition~4.5]{Latuszynski2013},
there exists $R>0$ such that, for all $N\in\N_{0}$,
\begin{align*}
\E_{\mu_1}[T^N[\sqrt{V}]] \le \E_{\mu_1}[T^N[V]]\le R,
\qquad
\bE_{P}[\sqrt{V}] \le \bE_{P}[V] \le R,
\end{align*}
where we used that $V \geq 1$.
\end{proof}

We are now in the position to formulate the main result of this section. 
\begin{theorem}
\label{thm:MarkovKDE}
    Under \Cref{ass:kde_rates} and \Cref{ass:Markov_chain} the KDE defined in \eqref{eq:kde_def_general} satisfies, for every $N\in\N$ and any bandwidth $0 < h \leq 1$,
    {\small
    \begin{align*}        
    \int_{\R^d} \E \bigg[\Big|\hat q_N(x)-\bE_{P}\big[K_h(x,\cdot)\big]\Big|\bigg]\d x
    &\le
    \frac{\alpha_0\|K\|_{L^1(\R^d)}}{\sqrt{N}} + \frac{\alpha_0 B_1}{\sqrt{N}h^{d}} + \frac{2\alpha_1\|K\|_{L^2(\R^d)}^2}{Nh^d} + \frac{2\alpha_1 B_2}{Nh^{2d}},
    \\    
    \int_{\R^d} \E \bigg[\Big(\hat q_N(x)-\bE_{P}\big[K_h(x,\cdot)\big]\Big)^{\! 2} \bigg]\d x
    &\le
    \frac{4\alpha_0^2 \|K\|_{L^2(\R^d)}^2}{Nh^d} + \frac{4\alpha_0^2 B_2}{Nh^{2d}} + \frac{8\alpha_1^2\|K\|_{L^4(\R^d)}^4}{N^2 h^{3d}}+\frac{8\alpha_1^2B_4}{N^2 h^{4d}}\,.    
    \end{align*}    
    } 
\end{theorem}

\begin{proof}
	The proof consists of applying \Cref{lemma:Latuszynski_summary} to $g = g_{x} \coloneqq K_h(x,\cdot)$ for fixed $x \in \bR^{d}$ and then integrating over $x$. Using that $K_h(x,\cdot)\in L^1(P)$ and $V\ge 1$, we obtain for $r=1,2,4$,
    \begin{align*}
         \int_{\R^d} \norm{\bE_{P}[K_h(x,\cdot)]}_{V^{1/2}}^r\,\rd x
         &=
         \int_{\R^d} \bE_{P}[K_h(x,\cdot)]^r\, \sup_{y\in\R^d} V^{-r/2}(y)\, \rd x
         \\
         &\le
         \int_{\R^d} \bE_{P}[K_h(x,\cdot)^{r}]\,\rd x
         \\
         &=
         h^{-rd} \bE_{P} \bigg[ \int_{\R^d} K\bigg( \frac{x - \cdot}{h} \bigg)^{\! r}\,\rd x \bigg]
        \\
        &=
        \frac{\|K\|_{L^r(\R^d)}^r}{h^{(r-1)d}}.
    \end{align*}
    Together with \Cref{ass:Markov_chain} (KMC) this gives
    \[
    \int_{\R^d}\|\bar g_x\|_{V^{1/2}}^r\, \rd x
    \leq
    r \int_{\R^d} \norm{K_h(x,\cdot)}^{r} + \norm{\bE_{P}[K_h(x,\cdot)]}_{V^{1/2}}^r\, \rd x
    \leq
    \frac{r\, \|K\|_{L^r(\R^d)}^r}{h^{(r-1)d}} + \frac{r\, B_r}{h^{r d}}\,.
    \]
    A direct application of \Cref{lemma:Latuszynski_summary} yields
    \begin{align*}
    \int_{\R^d} \E \bigg[\Big|\hat q_N(x)-\bE_{P}\big[ & K_h(x,\cdot)\big]\Big|\bigg]\d x    
    \le
    \int_{\R^d}
    \frac{\alpha_{0}\, \|\bar{g_{x}}\|_{V^{1/2}}}{\sqrt{N}}
    +
    \frac{\alpha_{1}\, \|\bar{g_{x}}\|_{V^{1/2}}^2}{N}\, 
    \d x
    \\
    &\le \frac{\alpha_0\|K\|_{L^1(\R^d)}}{\sqrt{N}} + \frac{\alpha_0 B_1}{\sqrt{N}h^{d}} + \frac{2\alpha_1\|K\|_{L^2(\R^d)}^2}{Nh^d} + \frac{2\alpha_1 B_2}{Nh^{2d}},
    \end{align*}
    and similarly,
    \begin{align*}
    \int_{\R^d} \E \bigg[\Big(\hat q_N(x)-\bE_{P}\big[ & K_h(x,\cdot)\big]\Big)^2\bigg]\d x
    \le
    \int_{\R^d}
    \bigg( \frac{\alpha_{0}\, \|\bar{g_{x}}\|_{V^{1/2}}}{\sqrt{N}}
    +
    \frac{\alpha_{1}\, \|\bar{g_{x}}\|_{V^{1/2}}^2}{N} \bigg)^{2}
    \d x
    \\
    &\le
    2 \int_{\R^d}
    \bigg( \frac{\alpha_{0}\, \|\bar{g_{x}}\|_{V^{1/2}}}{\sqrt{N}} \bigg)^{2}
    +
    \bigg(\frac{\alpha_{1}\, \|\bar{g_{x}}\|_{V^{1/2}}^2}{N} \bigg)^{2}
    \d x
    \\
    &\le \frac{4\alpha_0^2 \|K\|_{L^2(\R^d)}^2}{Nh^d} + \frac{4\alpha_0^2 B_2}{Nh^{2d}} + \frac{8\alpha_1^2\|K\|_{L^4(\R^d)}^4}{N^2 h^{3d}}+\frac{8\alpha_1^2B_4}{N^2 h^{4d}}\,.
    \end{align*}
\end{proof}

As a consequence of \Cref{thm:MarkovKDE}, we obtain the following KDE error bounds which are more conservative than the ones in the i.i.d.\ case (cf.\ \Cref{cor:kde_rates_stationary}).
\begin{corollary}[KDE error bounds based on non-stationary Markov chain data]
\label{thm:kde_rates}
Suppose \Cref{ass:kde_rates} holds and let $Z_1,\dots,Z_N$ be generated by a Markov chain satisfying \Cref{ass:Markov_chain}.   
Then the KDE $\hat{q}_{N}$ in \eqref{eq:kde_def_general} satisfies
\begin{align}
\label{MIAE_and_MISE_nonstationary}
&\MIAE(\hat q_N) = \mathcal O\Big(\frac{1}{\sqrt{Nh^{2d}}} + h^2\Big)&
&\text{and}&
&\MISE(\hat q_N) = \mathcal O\Big(\frac{1}{Nh^{2d}} + h^4\Big).&
\intertext{
Balancing the worst case upper bounds on bias and variance terms yields the asymptotically optimal bandwidth
$h_N \asymp N^{-1/(2d+4)}$,
and the corresponding rates}
\label{MIAE_and_MISE_nonstationary_optimal_h}
&\MIAE(\hat q_N) = \mathcal O \Big( N^{-1/(d+2)} \Big)&
&\text{and}&
&\MISE(\hat q_N) = \mathcal O \Big( N^{-2/(d+2)} \Big).&   
\end{align}
\end{corollary}

\begin{proof}
This follows directly from \Cref{lemma:KDE_bias} and \Cref{thm:MarkovKDE}.
\end{proof}

\section{Importance sampling with KDE proposals}
\label{sec:kde_is}

We now combine the KDE error bounds from \Cref{sec:kde_mcmc_mise} with the general IS results of \Cref{thm:basic_random_is,thm:variance_bound_for_safe_is,thm:error_bound_for_clipped_safe_snis}, translating them into explicit guarantees for IS with KDE proposals, and also making precise the joint roles of the proposal sample size $N$ and the importance sampling size $n$. While the underlying error bounds in \Cref{thm:basic_random_is,thm:variance_bound_for_safe_is,thm:error_bound_for_clipped_safe_snis}
are non-asymptotic, the $\MIAE$ and $\MISE$ bounds in \Cref{cor:kde_rates_stationary,thm:kde_rates} are not, hence the resulting rates here are stated using \( \mathcal O(\cdot)\)-notation,
with constants independent of \(N\) and \(n\).
\begin{theorem}
\label{thm:is_kde_final_rates}
Suppose \Cref{ass:kde_rates} holds and let $\hat q_N$ denote the KDE constructed from auxiliary samples $Z_1,\dots,Z_N$ as in \eqref{eq:kde_def_general}.
\begin{enumerate}[label = (\alph*)]
\item
\label{item:thm_is_kde_final_rates_stationary}
Suppose that \Cref{ass:spectral_gap} holds and that $h_N \asymp N^{-1/(d+4)}$.
    \begin{enumerate}[label = (\roman*)]
    \item
    \label{item:thm_is_kde_final_rates_stationary_nondefensive}
    If $\hat{q} = \hat{q}_{N}$ and \Cref{assump:basic_is} and the polynomial tail condition \eqref{eq:poly_tail_p_and_weighted_K2_moment} hold, then the IS estimator $\I_n(f)$ in \eqref{equ:IS_estimator} satisfies
    \begin{equation*}        
        \E\bigl[|\I_n(f)-\I(f)|\bigr]
        =
        \mathcal{O} \bigl(n^{-1/2} + N^{-1/(d+4)}\bigr).
    \end{equation*}

    \item
    \label{item:thm_is_kde_final_rates_stationary_defensive}
    If \Cref{assump:N_dependent_safe_is} holds and $\delta_{N} \asymp N^{-1/(d+4)}$, then
    each of the estimators $\I_n^{\ast} \in \{ \I_n,\,  \I^{\mathrm{SN}}_{n},\, \I^{\mathrm{SN},\tau}_{n} \}$ defined in \eqref{equ:IS_estimator} and \eqref{eq:selfnormIS} satisfies
    \begin{equation*}
    \E\bigl[|\I_n^{\ast}(f)-\I(f)|\bigr]
    \leq
    \frac{1}{\sqrt{n}} \left(\Var[f(Z)]^{1/2} + \mathcal{O}\big( N^{-1/(2d+8)} \big)\right),
    \end{equation*}
    where for $\I_n^{\ast} = \I^{\mathrm{SN}}_{n}$ we additionally require $\bE\big[ \norm{f}_{L^4(\hat{Q})}^{4} \big]$ to be finite.
    \end{enumerate}
    
\item
\label{item:thm_is_kde_final_rates_nonstationary}
Suppose that \Cref{ass:Markov_chain} holds and that $h_N \asymp N^{-1/(2d+4)}$.
    \begin{enumerate}[label = (\roman*)]
    \item
    \label{item:thm_is_kde_final_rates_nonstationary_nondefensive}
    If $\hat{q} = \hat{q}_{N}$ and \Cref{assump:basic_is}  and the polynomial tail condition \eqref{eq:poly_tail_p_and_weighted_K2_moment} hold, then the IS estimator $\I_n(f)$ in \eqref{equ:IS_estimator} satisfies
    \begin{equation*}        
        \E\bigl[|\I_n(f)-\I(f)|\bigr]
        =
        \mathcal{O} \bigl(n^{-1/2} + N^{-1/(2d+4)}\bigr).
    \end{equation*}

    \item
    \label{item:thm_is_kde_final_rates_nonstationary_defensive}
    If \Cref{assump:N_dependent_safe_is} holds and $\delta_{N} \asymp N^{-1/(2d+4)}$, then
    each of the estimators $\I_n^{\ast} \in \{ \I_n,\,  \I^{\mathrm{SN}}_{n},\, \I^{\mathrm{SN},\tau}_{n} \}$ defined in \eqref{equ:IS_estimator} and \eqref{eq:selfnormIS} satisfies
    \begin{equation*}
    \E\bigl[|\I_n^{\ast}(f)-\I(f)|\bigr]
    \leq
    \frac{1}{\sqrt{n}} \left(\Var[f(Z)]^{1/2} + \mathcal{O}\big( N^{-1/(4d+8)} \big)\right),
    \end{equation*}
    where for $\I_n^{\ast} = \I^{\mathrm{SN}}_{n}$ we additionally require $\bE\big[ \norm{f}_{L^4(\hat{Q})}^{4} \big]$ to be finite.
    \end{enumerate}    
\end{enumerate}
\end{theorem}

\begin{proof}
This follows directly from \Cref{thm:basic_random_is}\ref{item:basic_random_is_error}, \Cref{cor:error_bound_for_N_dependent_safe_is_and_snis,cor:kde_rates_stationary,thm:kde_rates}.
\end{proof}

\begin{table}[!t]
\centering
\caption{Summary of convergence rates stated in \Cref{thm:is_kde_final_rates} for the errors $\E\bigl[|\I_n^{\ast}(f)-\I(f)|\bigr]$ of KDE-based IS estimators $\I_n^{\ast} \in \{ \I_n,\,  \I^{\mathrm{SN}}_{n},\, \I^{\mathrm{SN},\tau}_{n} \}$ defined in \eqref{equ:IS_estimator} and \eqref{eq:selfnormIS} under the general \Cref{ass:kde_rates} with the KDE $\hat{q}_{N}$ defined by \eqref{eq:kde_def_general}.
Note that the rates for $\I_n^{\ast} = \I^{\mathrm{SN}}_{n}$ additionally require $\bE\big[ \norm{f}_{L^4(\hat{Q})}^{4} \big]$ to be finite.}
\label{tab:kde_is_rates}
\begin{tabularx}{\textwidth}{>{\raggedright\arraybackslash}p{4.0cm} XX}
\toprule
\makecell[l]{Random proposal $\hat{q}$
\\
and assumptions}
&
\makecell[c]{$Z_{k} \stackrel{\textup{i.i.d.}}{\sim} P$ or stationary MC
\\
(\Cref{ass:spectral_gap})
\\
with $h_N,\delta_{N}  \asymp N^{-1/(d+4)}$}
&
\makecell[c]{$(Z_{k})_{k \in \bN}$ non-stationary MC
\\
(\Cref{ass:Markov_chain})
\\
with $h_N,\delta_{N} \asymp N^{-1/(2d+4)}$}
\\
\midrule
\addlinespace[1em]
\makecell[l]{Non-defensive KDE\\
$\hat q = \hat q_N$, 
$\I_n^{\ast}  = \I_n$
\\
\Cref{assump:basic_is} + \eqref{eq:poly_tail_p_and_weighted_K2_moment}}
&
\makecell[l]{
$\mathcal O\!\left(
n^{-1/2} + N^{-1/(d+4)}
\right)$
}
&
\makecell[l]{
$\mathcal O\!\left(
n^{-1/2} + N^{-1/(2d+4)}
\right)$
}
\\
\addlinespace[1em]
\makecell[l]{Defensive KDE mixture
\\
$\hat q = \hat q_{N,\delta_N}$
\\
(\Cref{assump:N_dependent_safe_is})
\\
$\I_n^{\ast} \in \{ \I_n,\,  \I^{\mathrm{SN}}_{n},\, \I^{\mathrm{SN},\tau}_{n} \}$
}
&
\makecell[l]{
$\displaystyle
\frac{1}{\sqrt n}
\Big[
\Var[f(Z)]^{\frac{1}{2}}
+
\mathcal O \Big(
N^{- \frac{1}{2d+8}}
\Big)
\Big]$
}
&
\makecell[l]{
$\displaystyle
\frac{1}{\sqrt n}
\Big[
\Var[f(Z)]^{\frac{1}{2}}
+
\mathcal O \Big(
N^{-\frac{1}{4d+8}}
\Big)
\Big]$
}
\\
\bottomrule
\end{tabularx}
\end{table}

\begin{remark}\label{rem:thm6_1}
The error bounds in \Cref{thm:is_kde_final_rates}\ref{item:thm_is_kde_final_rates_stationary}\ref{item:thm_is_kde_final_rates_stationary_nondefensive} and \ref{item:thm_is_kde_final_rates_nonstationary}\ref{item:thm_is_kde_final_rates_nonstationary_nondefensive} suggest how to split the computational budget if the costs of adding another sample $Z_{k}$ and adding another sample $X_{i}$ are known (including the additional cost in the evaluation of the IS estimator $I_{n}$).
For example, if $Z_{k}$ are generated by a Metropolis--Hastings algorithm and the costly step is the evaluation of the target density $p$, then both costs are comparable, suggesting to choose
$n \asymp N^{2/(d+4)}$ under \Cref{ass:spectral_gap}
and $n \asymp N^{1/(d+2)}$ under
\Cref{ass:Markov_chain}.
However, this allocation is rather inefficient: Since $n$ grows only slowly compared to $N$, most of the computational budget is spent on generating auxiliary samples for the KDE proposal, while relatively few samples contribute to the actual IS estimator.
This observation strongly suggests the defensive construction considered in \Cref{thm:is_kde_final_rates}\ref{item:thm_is_kde_final_rates_stationary}\ref{item:thm_is_kde_final_rates_stationary_defensive} and \ref{item:thm_is_kde_final_rates_nonstationary}\ref{item:thm_is_kde_final_rates_nonstationary_defensive}, which allows the estimator to retain the optimal Monte Carlo rate $n^{-1/2}$ while the proposal error only appears as a mild variance inflation term.
\end{remark}

\section{Numerical illustration}
\label{sec:numerics}
The following experiment illustrates the behavior of the proposed defensive KDE-based IS scheme on a simple one-dimensional example.
The target distribution $P$ is a shifted Laplace distribution given by its probability density function
\[
p(x)
\coloneqq
\frac{1}{2b}\exp\Big(-\frac{|x-\mu|}{b}\Big)\,,
\qquad
\mu=10,\ b=1.5,\ x\in\R\,.
\]
Our defensive distribution is defined by a centered Cauchy distribution with scale parameter $s=30$, 
\[ \varphi(x) = \frac{s}{\pi(s^2+x^2)}\,,\quad x\in\R\,.\]
The defensive proposal is constructed as
\( \hat q_{N,\delta}(x) = (1-\delta)\hat q_N(x) + \delta \varphi(x)\,,\)
where $\hat q_N$ is a Gaussian KDE constructed from i.i.d.~samples $Z_1,\dots,Z_N\sim p$ with bandwidth $h_N = \Var_P(Z)^{\frac12} N^{-1/5}$, and $\delta = \delta_{N}$ may depend on the number of auxiliary samples.
The variance of $Z\sim P$ is given by $\Var_P(Z)=4.5$. 

We consider the estimation of the expectation $\mathcal I(f)=\E_P[f]$ for $f(x) = (x-5)^2$ using the IS estimator ${\mathcal I}_n(f)$. The exact value is $\mathcal I(f) = 29.5$. 
In our experiment, we compare three different choices of mixture weight $\delta$. 
Specifically, we consider $\delta\in\{0,0.5,1\}$, where $\delta=0$ corresponds to a pure KDE proposal and $\delta=1$ to a pure defensive proposal distribution.
In addition, we consider the theoretically motivated schedule $\delta_N=0.5 \cdot N^{-1/5}$ as suggested by the non-asymptotic bounds in
\Cref{thm:is_kde_final_rates}\ref{item:thm_is_kde_final_rates_stationary}.

\begin{figure}[!t]
  \centering \includegraphics[width=0.6\textwidth]{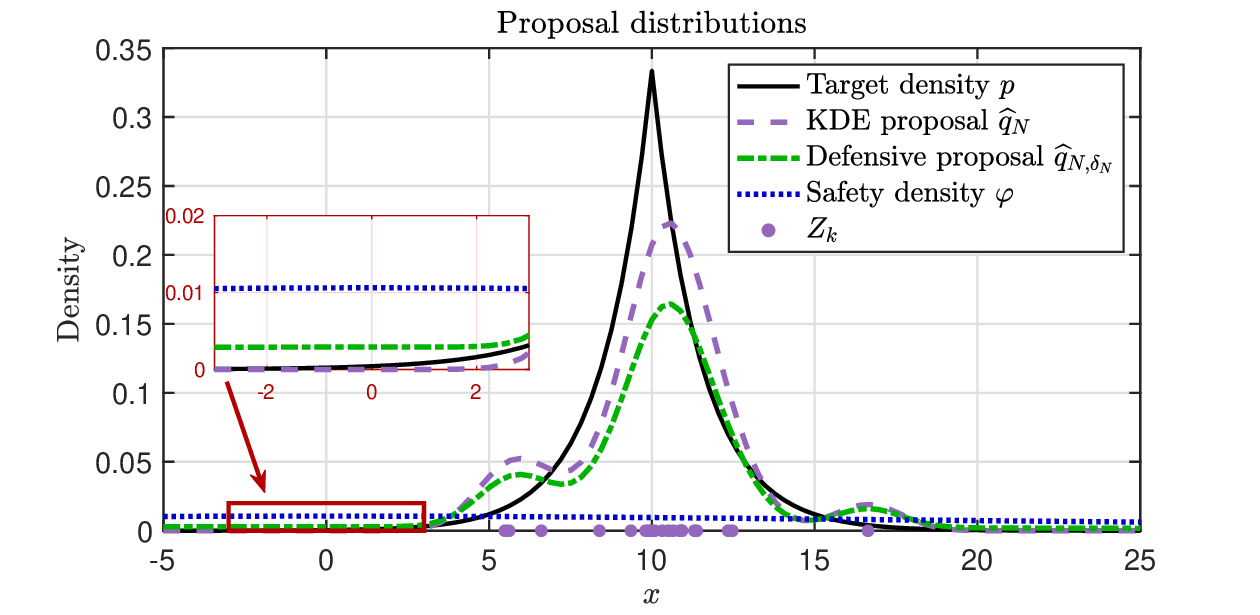}
 \caption{Illustration of the target probability density function and proposal densities.} \label{fig:kde}
\end{figure} 

\begin{figure}[!htb]
  \centering 
  \includegraphics[width=1\textwidth]{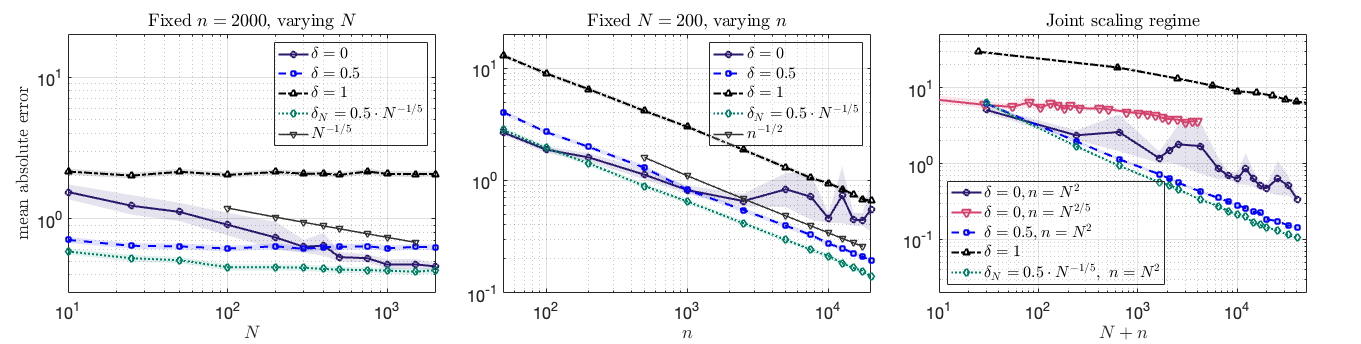}
 \caption{Left: Performance of the defensive KDE IS estimator with fixed IS sample size $n=2000$ depending on reference sample size $N$. Middle: Performance of the defensive KDE IS estimator with fixed reference sample size $N=200$ depending on IS sample size $n$. Right: Performance of the defensive KDE IS estimator under joint scaling regimes.
 All plots show the MAE for different choices of $\delta$ (as functions of $N$) and different scaling regimes.
 Shaded regions indicate bootstrap $95\%$ uncertainty bands for the estimated MAE, based on the 1000 Monte Carlo replications.} \label{fig:varyingsmalln}
\end{figure} 

The experiment is decomposed into three regimes. First, we consider a fixed importance sample size $n=2000$ and vary the reference sample size $N$. Second, we fix the reference sample size $N=200$ and vary the importance sample size $n$. Third, we consider a coupled regime $n \asymp N^2$ to illustrate the stabilizing effect of the defensive schedule $\delta_N$. Additionally, we also include the regime $n\asymp N^{2/5}$ that balances $n$ and $N$ in the upper bound $\mathcal O(n^{-1/2}+N^{-1/5})$ of the MAE in \Cref{thm:is_kde_final_rates}\ref{item:thm_is_kde_final_rates_stationary}\ref{item:thm_is_kde_final_rates_stationary_nondefensive}, see \Cref{rem:thm6_1}.
For each configuration the 
mean absolute error (MAE) $\bE[ | \mathcal I_{n}(f) - \mathcal I(f) | ]$ is estimated over $1000$ independent replications. To quantify the Monte Carlo variability of these performance metrics, \Cref{fig:varyingsmalln} reports the estimated MAE including $95\%$ confidence bands estimated by bootstrap.
The results highlight the effect of the defensive mixture approach. When $\delta = 0$, the proposal is entirely based on the KDE and leads to unstable estimators when the KDE underestimates the target density $p$ in the tails. This can also be observed in the regime $n\asymp N^{2/5}$ where the non-defensive IS estimator suffers from the slow convergence of the KDE. In contrast, a moderate value of $\delta$ stabilizes the estimator. The scheduled choice $\delta_N$ successfully interpolates between the pure defensive and full KDE proposal, providing a robust estimation performance.

\begin{appendix}

\section{Universal Sampling and Conditioning on a Random Measure}
\label{sec:universal_sampling}

The following result justifies the identity $\E [g(X) \, | \,  \hat Q]=\int g\,\d\hat Q$, where $\hat{Q}$ is a random probability measure and $X \sim \hat{Q}$.
To ensure this equality holds, we construct $X$ as a $\hat{Q}$–measurable transform of an auxiliary random variable $U$ that is independent of $\hat{Q}$ (i.e., $X=\Psi(\hat{Q},U)$ for a suitable Borel map $\Psi$).
As usual, $\cP(\bR^{d})$ denotes the space of probability measures on $(\bR^{d},\cB(\bR^{d}))$, equipped with the Borel $\sigma$-algebra of its weak topology.

\begin{theorem}[Universal sampling and conditioning identity]
\label{thm:universal_sampling}
There exists a Borel-measurable map $\Psi\colon \cP(\bR^{d}) \times [0,1] \to  \bR^{d}$ such that, for every $\mu\in\cP(\bR^{d})$ and $U\sim\mathrm{Unif}[0,1]$,
\(
\Psi(\mu,U)\sim \mu.
\)
Let $\hat Q$ be a $\cP(\bR^{d})$-valued random variable on $(\Omega,\Sigma,\P)$, let $U\sim\mathrm{Unif}[0,1]$ be independent of $\hat Q$ (by enlarging the space if necessary) and consider the $\bR^{d}$-valued random variable $X \coloneqq \Psi(\hat Q,U)$.
Then its conditional law given $\hat{Q}$ satisfies $\mathcal L(X \, | \, \hat Q) = \hat Q$ a.s., and, for every measurable $g\colon \bR^{d}\to\R$ with $\E \big[\int \abs{g} \,\d\hat Q \big] < \infty$,
\[
\bE \big[g(X) \, |\, \hat Q \big]
=
\int g(x)\,\hat Q(\d x)\quad\text{a.s.}
\]
In particular, $\bE[|g(X)|]<\infty$.
\end{theorem}

\begin{proof}
\emph{Step 1 (The kernel and universal sampler).}
By the kernel criteria, $(\mu,B)\mapsto \mu(B)$ is a probability kernel from $\cP(\bR^{d})$ to $\bR^{d}$.
Consequently, for every bounded measurable $g\colon \bR^{d}\to\R$, the map $\mu\mapsto\int g\,\d\mu$ is measurable on $\cP(\bR^{d})$ \cite[Lemma~3.1]{Kallenberg2021}.
By the kernel representation lemma, there exists a Borel map $\Psi:\cP(\bR^{d})\times[0,1]\to \bR^{d}$ such that $\Psi(\mu,U)\sim \mu$ for $U\sim\mathrm{Unif}[0,1]$ \cite[Lemma~4.22]{Kallenberg2021}.
If the probability space does not carry an independent $U$, enlarge it; see \cite[Lemma~8.16, Thm.~8.17]{Kallenberg2021}.

\emph{Step 2 (Conditional law).}
Fix $\nu\in\cP(\bR^{d})$. By construction, $\Psi(\nu,U)\sim\nu$, and therefore $\E[g(\Psi(\nu,U))]=\int g\,\d\nu$ for every bounded measurable $g$.
Substituting $\nu=\hat Q(\omega)$ and using independence of $U$ and $\hat Q$,
\[
\E \big[ g \big(\Psi(\hat Q,U)\big) \,|\, \hat Q\big]
=
\big(\nu\mapsto \E[g(\Psi(\nu,U))]\big)\Big|_{\nu=\hat Q}
=
\int g\,\d\hat Q\quad\text{a.s.}
\]
Thus $\mathcal L(X\mid \hat Q)=\hat Q$ and the identity holds for bounded $g$.

\emph{Step 3 (Extension to $g\ge 0$).}
For measurable $g\ge 0$, applying Step~2 to $g_n \coloneqq \min(g,n)$ and using monotone convergence twice,
\[
\E\big[ g(X) \, | \, \hat Q \big]
=
\lim_{n\to\infty} \E \big[ g_n(X) \, | \, \hat Q \big]
=
\lim_{n\to\infty}\int g_n\,\d\hat Q
=
\int g\,\d\hat Q\quad\text{a.s.}
\]

\emph{Step 4 (Extension to integrable $g$).}
For a measurable $g\colon \R^d\to\R$ write $g^{+}=\max\{g,0\}$ and $g^{-}=\max\{-g,0\}$, so that $g=g^{+}-g^{-}$ and $|g|=g^{+}+g^{-}$.
If $\E \big[\int |g|\,\d\hat Q \big]<\infty$, then in particular
$\E \big[\int g^{\pm}\,\d\hat Q \big]<\infty$, hence $\int g^{\pm}\,\d \hat Q<\infty$ a.s.
By Step~3 applied to $g^{+}$ and $g^{-}$,
$
\E\big[g^{\pm}(X)\,|\,\hat Q\big]=\int g^{\pm}\,\d \hat Q
$
a.s.
Subtracting/adding the identities for $g^{+}$ and $g^{-}$ yields
\[
  \E \big[g(X)\,|\,\hat Q\big]=\int g\,\d \hat Q \quad \text{a.s.},
  \qquad
  \E \big[\abs{g(X)}\,|\,\hat Q\big]=\int \abs{g}\,\d \hat Q \quad \text{a.s.}
\]
Finally, taking expectations and using $\E \big[\int |g|\,\d\hat Q \big]<\infty$ gives $\E[ |g(X)| ]<\infty$.
\end{proof}

\begin{remark}[Realization on an extended space; w.l.o.g.\ representation]
\label{rem:wlog_realization_X}
In applications we may, \emph{without loss of generality}, assume that our sampling variable is of the form
$
X=\Psi(\hat Q,U)
$
for a Borel map $\Psi \colon \cP(\R^d)\times[0,1]\to\R^d$ and an auxiliary $U\sim\mathrm{Unif}[0,1]$ independent of $\hat Q$ (after enlarging the probability space if needed).
Indeed, given any pair $(\hat Q,X)$ with $\mathcal{L}(X\,|\, \hat Q)=\hat Q$ a.s., there exists an extension of the space, an independent $U$, and a Borel $\Psi$ such that a copy $X'$ of $X$ satisfies
\[
X'=\Psi(\hat Q,U)\quad\text{a.s.},\qquad (\hat Q,X')\stackrel{d}=(\hat Q,X).
\]
All conditional--expectation identities we use (e.g.\ $\E[g(X)\,|\, \hat Q]=\int g\,\d \hat Q$ under the stated integrability) depend only on the \emph{joint law} of $(\hat Q,X)$; hence replacing $X$ by $X'$ is legitimate and we may work with the representation $X=\Psi(\hat Q,U)$ throughout.
\end{remark}

\section{Omitted Proofs}
\label{sec:proofs}

\begin{proof}[Proof of \Cref{prop:lowerbound}]
    Consider the parametrized class of bimodal probability distributions $(P_{n})_{n \in \bN}$ defined via their densities
    \[
    p_{n}
    =
    \tfrac12 \, \varphi_{-1,\sigma_{n}} + \tfrac12 \, \varphi_{3,\sigma_{n}}\,,
    \qquad
    \varphi_{m,s^{2}}(x)
    \coloneqq
    (2 \pi s^2)^{-1/2}\, \exp\!\bigg( - \frac{(x-m)^{2}}{2 s^2} \bigg),
    \]
    where $\sigma_{n} \coloneqq (4\log(n+1))^{-\frac12} < 1$.
    Since $f(x) = x$ we obtain
    \[
    I(f) \coloneqq \bE_{P_{n}}[f] = 1,
    \qquad
    \|f\|_{L^2(P_{\bar n})}^2
    =
    \sigma_{n}^2 + \frac12 ( 1^{2} + 3^{2})
    \leq
    6,    
    \]
    independent of $n \in \bN$.
    We consider sampling a random variable $Z$ from $P_{n}$ by $Z = -1 + 4V + Y$ with Bernoulli distributed $V \sim \mathrm{Ber}(1/2)$ and independent normally distributed $Y \sim \cN\big( 0 , \sigma_{n}^{2}\big)$.
    Let $Z_{i} = -1 + 4V_{i} + Y_{i}$, $i = 1,\dots,N$, be independent realizations of this procedure and define the event
    $E_{N} \coloneqq \{ V_{1} = \cdots = V_{N} = 0\}$
    (`all samples come from the left component of $P_{n}$'), which clearly has probability $\P(E_N) = 2^{-N} = \exp(-N\log(2))$.   
    Further, define the event
    \[
    F_{N,n}
    \coloneqq
    \{\hat V(Z_1,\dots,Z_N) < \sigma_{n}^{2}\}.
    \]
    Since, conditional on $E_{N}$, the random variable $S_{N-1} \coloneqq \frac{N-1}{\sigma_{n}^{2}} \hat V(Z_1,\dots,Z_N)$ is $\chi_{N-1}^{2}$ distributed,
    \begin{equation}
    \label{equ:prob_of_E_N_cap_F_Nn}
    \P(E_N\cap F_{N,n}) = \P \big(S_{N-1} < N-1 \mid E_N \big) \, \P(E_N)
    \ge
    \tfrac{1}{2}\exp(-N\log(2))\,.
    \end{equation}
    Conditional on $E_N\cap F_{N,n}$, the samples $X_{i}$ are i.i.d.\ with distribution
    $N(-1,\sigma_{n}^2+h_N^2)$, and $\sigma_{n}^2+h_N^2 \leq 2\sigma_{n}^2$ by the definition of $F_{N,n}$.
    Denote by $\Phi$ the standard normal cumulative distribution function and recall that, by Mill's ratio \citep{Wainwright_2019}, $\Phi(-x)\le \varphi_{0,1}(x)/x$ for all $x>0$.
    Hence, by the definition of $\sigma_{n} \coloneqq (4\log(n+1))^{-\frac12} < 1$ and setting $c \coloneqq (4 \pi \log 2)^{-1/2} < 1$,
    \begin{equation}
    \label{equ:technical_inequality_using_Mills_ratio}
    \begin{alignedat}{3}
    1 - \P(X_i \le 0 \mid E_N\cap F_{N,n})
    \ &=\ 
    1 - \Phi\big( (\sigma_{ n}^2 + h_{N}^{2})^{-1/2} \big)
    &
    \ &\leq\ 
    \Phi\big( - (2\sigma_{n}^{2})^{-1/2} \big)
    \\
    \ &\leq\ 
    \sqrt{2}\sigma_{n} \varphi_{0,1}\big( (2\sigma_{n}^{2})^{-1/2} \big)
    &
    \ &=\ 
    \frac{\sigma_{n}}{\sqrt{\pi}} \exp \big( - (2\sigma_{n})^{-2} \big)
    \\
    \ &\leq\ 
    \frac{c}{n+1}.
    \end{alignedat}
    \end{equation}
    Consequently,
    \begin{flalign*}
    \hspace{1em}
    \bE \big[ | I_{n}(f) - I(f) | \big]
    &\geq \bP\big( \abs{I_{n}(f) - I(f)} \geq 1 \big)
    & \text{Markov's inequality}&
    \\
    &\geq \bP\big( I_{n}(f) \leq 0 \big)
    & \text{$I(f) = 1$}&
    \\
    &\geq \bP( X_i \leq 0 \ \forall i=1,\dots,n )
    & \text{definition of $I_n$ and $f$}&
    \\
    &= \bP(E_N\cap F_{N,n}) \, \bP( X_1 \leq 0 \mid E_N\cap F_{N,n} )^{n}
    & \text{conditional independence}&
    \\
    &\geq \tfrac{1}{2}\exp(-N\log(2))\, \bigg( 1 - \frac{c}{n+1} \bigg)^{n+1}
    & \text{by \eqref{equ:prob_of_E_N_cap_F_Nn} and \eqref{equ:technical_inequality_using_Mills_ratio}}&
    \end{flalign*}
    Since \(0<c<1\), we have \(1-c/(n+1)>0\) for all \(n\in\mathbb N\). Moreover,
    $(1 - c/(n+1))^{n+1} \to e^{-c}>0$ as $n \to \infty$.
    Hence this sequence is bounded away from zero uniformly in \(n\). Consequently, there exists a constant \(C_N>0\), depending only on \(N\), such that
    \[
    \mathbb E\bigl[|I_n(f)-I(f)|\bigr]
    \ge C_N
    \]
    for all \(n\in\mathbb N\), which proves the claim.
\end{proof}

\begin{proof}[Proof of \Cref{lemma:KDE_bias}]
	Denote
	\begin{align*}
	(K_h*p)(x)
	&\coloneqq
	\int_{\R^d} K_h(x,y)\,p(y)\,\rd y
	=
	\int_{\R^d} K(u)\,p(x-hu)\,\rd u ,
	\\ 
	M_{2}(K)
	&\coloneqq
	\int_{\R^d} \|u\|^{2}\, |K(u)|\,\rd u.
	\end{align*}
	Since \(p\in W^{2,1}(\R^d)\cap W^{2,2}(\R^d)\), we cannot directly apply the classical pointwise Taylor formula to \(p\).
	We therefore argue by mollification.
	Let $(\rho_\varepsilon)_{\varepsilon>0}$ be a standard mollifier, i.e.
	$\rho_\varepsilon(x) = \varepsilon^{-d}\rho(x/\varepsilon)$, where
	$\rho \in C_c^\infty(\R^d)$, $\rho \ge 0$, and $\int_{\R^d} \rho(x)\,dx = 1$, and define
	$p_\varepsilon \coloneqq p * \rho_\varepsilon$.
	Then \(p_\varepsilon \in C^\infty(\R^d)\), and for every \(i,j\in\{1,\dots,d\}\) and \(q\in\{1,2\}\),
	\[
	\partial_{ij}p_\varepsilon
	=
	(\partial_{ij}p)*\rho_\varepsilon,
	\qquad
	\|\partial_{ij}p_\varepsilon\|_{L^q(\R^d)}
	\le
	\|\partial_{ij}p\|_{L^q(\R^d)}
	\]
	by Young's inequality. Moreover,
	$p_\varepsilon \to p$ in $L^1(\R^d)$ and in $L^2(\R^d)$ as $\varepsilon\to 0$.	
	We first prove the desired bounds for the smooth density \(p_\varepsilon\).
	By the second-order Taylor formula in integral form,
	\[
	p_\varepsilon(x-hu)
	=
	p_\varepsilon(x)
	-
	h\,u\cdot \nabla p_\varepsilon(x)
	+
	h^2\int_0^1 (1-t)\, u^\top D^2p_\varepsilon(x-thu)\,u\,\rd t .
	\]
    Since the kernel $K$ has vanishing first moment by assumption, it follows that
	\begin{align*}
		(K_h*p_\varepsilon)(x) - p_\varepsilon(x)
		&=
		\int_{\R^d} K(u)\, \big( p_\varepsilon(x-hu) - p_\varepsilon(x) \big) \,\rd u
		\\
		&=
		h^2\int_{\R^d} K(u)\int_0^1 (1-t)\,
		u^\top D^2p_\varepsilon(x-thu)\,u\,\rd t\,\rd u
		\\
		&\le
		h^2\int_{\R^d} |K(u)|\,\|u\|^2
		\int_0^1 (1-t)\sum_{i,j=1}^d |\partial_{ij}p_\varepsilon(x-thu)|\,\rd t\,\rd u ,
	\end{align*}
	where we used
	\[
	|u^\top D^2p_\varepsilon(z)\,u|
	\le
	\|u\|^2 \sum_{i,j=1}^d |\partial_{ij}p_\varepsilon(z)| .
	\]
	Let \(q\in\{1,2\}\). Using Fubini's theorem for \(q=1\), and Minkowski's
	integral inequality for \(q=2\), together with translation invariance of the
	\(L^q\)-norm, we obtain
	\begin{align*}
		\|K_h*p_\varepsilon-p_\varepsilon\|_{L^q(\R^d)}
		&\le
		h^2 \int_{\R^d} |K(u)|\,\|u\|^2
		\int_0^1 (1-t) \,
		\bigg\|
		\sum_{i,j=1}^d |\partial_{ij}p_\varepsilon(\cdot-thu)|
		\bigg\|_{L^q(\R^d)}
		\rd t\,\rd u
		\\
		&\leq
		h^2 \int_{\R^d} |K(u)|\,\|u\|^2\,\rd u
		\int_0^1 (1-t)\,\rd t \, 
		\sum_{i,j=1}^d \|\partial_{ij}p\|_{L^q(\R^d)}
		\\
		&\leq
		\frac{h^2}{2}\,M_2(K)\sum_{i,j=1}^d \|\partial_{ij}p\|_{L^q(\R^d)} .
	\end{align*}
	It remains to pass to the limit \(\varepsilon\to0\).
	Since \(K\in L^1(\R^d)\),
	Young's convolution inequality yields	
	\begin{align*}
		\|K_h*p-p-(K_h*p_\varepsilon-p_\varepsilon)\|_{L^q(\R^d)}
		&\le
		\|K_h*(p-p_\varepsilon)\|_{L^q(\R^d)}
		+
		\|p-p_\varepsilon\|_{L^q(\R^d)}
		\\
		&\le
		\|K_h\|_{L^1(\R^d)}\|p-p_\varepsilon\|_{L^q(\R^d)} + \|p-p_\varepsilon\|_{L^q(\R^d)}
		\to 0.
	\end{align*}
	Therefore,
	\[
	\|K_h*p-p\|_{L^q(\R^d)}
	\le
	\frac{h^2}{2}\,M_2(K)\sum_{i,j=1}^d \|\partial_{ij}p\|_{L^q(\R^d)}\, ,
	\]	
	proving the claim with
	\[
	C_{\mathrm{bias},1}
	=
	\frac{M_2(K)}{2} \sum_{i,j=1}^d \|\partial_{ij}p\|_{L^1(\R^d)},
	\qquad
	C_{\mathrm{bias},2}
	=
	\frac{M_2(K)^2}{4}\,
	\bigg(\sum_{i,j=1}^d \|\partial_{ij}p\|_{L^2(\R^d)}\bigg)^{\! 2}.
	\]
\end{proof}

\begin{proof}[Proof of \Cref{thm:MarkovKDE_stationary}]
	Fix $x\in\R^d$ and define
	\[
	g_x(z) \coloneqq K_h(x,z),
	\qquad
	\bar g_x(z) \coloneqq g_x(z)-\bE_P[g_x].
	\]	
	Then
	\[
	\hat q_N(x)-\bE_P[K_h(x,\cdot)]
	=
	\frac1N\sum_{k=1}^N \bar g_x(Z_k).
	\]	
	By \cite[Theorems~3.1 and~3.2]{Paulin2015concentration}, there exists a constant $C_{\mathrm{mc}} > 0$ such that
	\begin{equation}
	\label{equ:pointwise_Paulin_application}
	\bE \left[
	\big(
	\hat q_N(x)-\bE_P[K_h(x,\cdot)]
	\big)^2
	\right]
	=
	\frac1{N^2}
	\Var \bigg(\sum_{k=1}^N g_x(Z_k)\bigg)
	\le
	\frac{C_{\mathrm{mc}}}{N}\Var_P(g_x)
	\le
	\frac{C_{\mathrm{mc}}}{N}\, \bE_P[g_x^2].
	\end{equation}
    In the case of i.i.d.\ data, the same result holds for $C_{\mathrm{mc}} = 1$.
	The \(L^2\)-bound \eqref{equ:thm_MarkovKDE_stationary_L2_bound} follows from integrating over $x$, since
	\begin{align*}
	\int_{\R^d}\bE_P[g_x^2] \, \rd x
	=
	\int_{\R^d}\int_{\R^d}K_h(x,z)^2\,P(\rd z) \, \rd x
	&=
	\int_{\R^d}
	\left(\int_{\R^d}K_h(x,z)^2\rd x\right)
	P(\rd z)\\
	&=
	\frac1{h^d}\|K\|_{L^2(\R^d)}^2.
	\end{align*}
	The \(L^1\)-bound \eqref{equ:thm_MarkovKDE_stationary_L1_bound} can be derived from the \eqref{equ:pointwise_Paulin_application} using the tail condition \eqref{eq:poly_tail_p_and_weighted_K2_moment}.
    First note that
	\[
	(1+\|x-y\|)(1+\|y\|)
	\geq
	1+\|x-y\| + \norm{y}	
	\geq
	1+\|x\|
	\]
	implies
	\[
	(1+\|x-y\|)^{-m}
	\le
	(1+\|x\|)^{-m}(1+\|y\|)^m.
	\]
	Hence, using \eqref{eq:poly_tail_p_and_weighted_K2_moment} and denoting $\kappa_{h}(x) \coloneqq K(x/h)^{2}$,
	\begin{align*}
	\int_{\R^d}
	\Bigl(\bE_P[g_x^2]\Bigr)^{1/2}\rd x
	&=
	\frac{1}{h^{d}}
	\int_{\R^d}\sqrt{(p*\kappa_h)(x)}\,\rd x
	\\
	&=
	\frac{1}{h^{d}}
	\int_{\R^d} \bigg(\int_{\R^d}  p(x-y)\, \kappa_h(y)\, \rd y \bigg)^{1/2}\,\rd x
	\\
	&\leq
	\frac{\sqrt{C_{p}}}{h^{d}}
	\int_{\R^d} \bigg( (1+\|x\|)^{-m}\int_{\R^d}  (1+\|y\|)^m\, \kappa_h(y)\, \rd y \bigg)^{1/2}\,\rd x
	\\
	&=
	\sqrt{\frac{C_{p}}{h^{d}}}
	\int_{\R^d}  (1+\|x\|)^{-m/2} \bigg( \int_{\R^d}  (1+h\|u\|)^m\, K^{2}(u)\, \rd u \bigg)^{1/2}\,\rd x
	\\
	&=
	\frac{\sqrt{C_p\,M_{K,m}}}{h^{d/2}}
	\int_{\R^d}(1+\|x\|)^{-m/2}\,\rd x,
	\end{align*}
	where we used that $0 < h < 1$ in the last step.
	Denoting $C_{1} \coloneqq \sqrt{C_{\mathrm{mc}}\,C_p\,M_{K,m}}
	\int_{\R^d}(1+\|x\|)^{-m/2}\,\rd x$ and applying Jensen's inequality (pointwise in \(x\)) to \eqref{equ:pointwise_Paulin_application} gives
	\[
	\int_{\R^d}
	\E\bigl[
	\bigl|\hat q_N(x)-\bE_P[K_h(x,\cdot)]\bigr|
	\bigr]
	\,\rd x
	\le
	\int_{\R^d}
	\sqrt{\frac{C_{\mathrm{mc}}}{N}\,
	\bE_P[g_x^2]}
	\,\rd x
	\leq
	\frac{C_{1}}{\sqrt{N}\,h^{d/2}}.
	\]
\end{proof}

\begin{proof}[Proof of \Cref{lem:weightedVnorm}]
	First note that, for fixed $x\in\R^d$,
	\begin{align*}
	&\|x-hz\|\ge \|x\|-h\|z\|\ge \frac{\|x\|}2&
	&\text{if } z \in A_x \coloneqq \Big\{z \in\R^d \mid \|z\|<\frac{\|x\|}{2h}\Big\},
	\\
	&\|z\|\ge\frac{\|x\|}{2h}&
	&\text{if } z \in B_x \coloneqq \Big\{z \in\R^d \mid \|z\|\ge \frac{\|x\|}{2h}\Big\}.
	\end{align*}
	Since $K(z)\le \Psi(\|z\|)\le \Psi(0)$ and $V^{-1/2}\le \Phi\le 1$, writing  $z=(x-y)/h$ gives
	\begin{align*}
		h^{rd}\, S_h^{r}(x)
		&=
		\sup_{z\in\R^d} K^{r}(z)V^{-r/2}(x-hz)
		\\
		&\leq
		\sup_{z\in A_{x}} \Psi^{r}(0)\, \Phi^{r}(\|x-hz\|)
		+
		\sup_{z\in B_{x}} \Psi^{r}(\norm{z})
		\\
		&\leq
		\Psi^{r}(0)\, \Phi^{r}(\|x\|/2)
		+
		\Psi^{r}(\norm{x}/(2h)).
	\end{align*}
	Integrating and applying the substitutions $x=2u$ and $x=2hu$ yields
	\begin{align*}
		h^{rd}\int_{\R^d}S_h(x)^r\,dx
		&\le
		\Psi(0)^r\int_{\R^d}\Phi(\|x\|/2)^r\,dx
		+
		\int_{\R^d}\Psi(\|x\|/(2h))^r\,dx
		\\
		&=
		2^d \Psi(0)^r \int_{\R^d}\Phi(\|x\|)^r\,dx
		+
		2^d h^d \int_{\R^d}\Psi(\|x\|)^r\,dx .
	\end{align*}
	Switching to polar coordinates and dividing by $h^{rd}$ completes the proof.
\end{proof}

\end{appendix}

\section*{Acknowledgments}
\addcontentsline{toc}{section}{Acknowledgements}
This research was funded by the Deutsche Forschungsgemeinschaft (DFG, German Research Foundation) --- CRC/TRR 388 ``Rough Analysis, Stochastic Dynamics and Related Fields'' --- Project ID 516748464. The authors would like to thank Jason Beh, Robert Scheichl, Claudia Schillings, Björn Sprungk, Lukas Trottner and Sven Wang for helpful discussions and suggestions.

\bibliographystyle{abbrvnat}
\bibliography{myBibliography} 
\addcontentsline{toc}{section}{References}

\end{document}